\documentclass[a4paper,onecolumn,
superscriptaddress,10pt, accepted=2020-05-14,issue=2,
volume=2,
shorttitle=papers/compositionality-2-2]{compositionalityarticle}
 \pdfoutput=1

\usepackage{amsmath,graphicx}
\usepackage[nocompress]{cite}

\usepackage{amsthm}
\usepackage{amssymb}
\usepackage{xypic}
\usepackage[square]{natbib}
\setcitestyle{numbers}
\theoremstyle{plain}
\newtheorem{theorem}{Theorem}
\newtheorem{proposition}{Proposition}
\newtheorem{corollary}{Corollary}
\newtheorem{lemma}{Lemma}
\theoremstyle{definition}
\newtheorem{definition}{Definition}

\newtheorem{remark}{Remark}
\newtheorem{example}{Example}

\def\st{\textrm{star }}
\def\shf{\mathcal}
\def\col{\mathcal}

\def\st{\textrm{star }}

\def\id{\textrm{id }}

\def\cat{\mathbf}

\title{Assignments to sheaves of pseudometric spaces}
\author{Michael Robinson}
\email{michaelr@american.edu}
\affiliation{Mathematics and Statistics,
American University,
Washington, DC, USA}
\begin{document}
\maketitle

\begin{abstract}
  An assignment to a sheaf is the choice of a local section from each open set in the sheaf's base space, without regard to how these local sections are related to one another.  This article explains that the \emph{consistency radius} --- which quantifies the agreement between overlapping local sections in the assignment --- is a continuous map.  When thresholded, the consistency radius produces the \emph{consistency filtration}, which is a filtration of open covers.  This article shows that the consistency filtration is a functor that transforms the structure of the sheaf and assignment into a nested set of covers in a structure-preserving way.  Furthermore, this article shows that consistency filtration is robust to perturbations, establishing its validity for arbitrarily thresholded, noisy data.
\end{abstract}

\tableofcontents

\section{Introduction}

An \emph{assignment} of a sheaf of sets on a topological space selects a local section from each open set in the sheaf's base space without regard to how these local sections are related to one another.  This article explores the structure of a \emph{sheaf of pseudometric spaces on a topological space} by analyzing the structure of its assignments.  More structure is present in a sheaf of pseudometric spaces than in a sheaf of sets, since each stalk is a set equipped with a \emph{pseudometric}, and each restriction is required to be continuous with respect to the pseudometrics present on its domain and codomain.

A sheaf of pseudometric spaces is best thought of as a tool for measuring the self-consistency between the local sections in an assignment.  When data are present in a \emph{global section} on two overlapping open sets, their consistency is manifest by being equal upon restricting to the intersection of the open sets.  For a sheaf of sets, an assignment is self-consistent when it is a global section, and that is all that can be said.  But for a sheaf of pseudometric spaces, instead of requiring equality, we can instead ask for the data to restrict to values within some distance of each other.  The supremum\footnote{or other suitable aggregate} of these distances is called the \emph{consistency radius} of the assignment.  Since global sections evidently have zero consistency radius --- the distance between data restricted to the same open set is always zero in a global section because they are equal --- positive consistency radius is an obstruction to an assignment being a global section.  

 It may happen that some parts of an assignment are more consistent than others, which is to say that restricting the consistency radius to subspaces of the sheaf's base space yields new information.  The resulting \emph{local consistency radius} can be used to identify open sets on which the assignment and sheaf are consistent to some specified level.   Sweeping a threshold value over the local consistency radius identifies portions of the base space on which an assignment is more or less in agreement with the restrictions of the sheaf.  Since increasing the threshold loosens the tolerance for data to restrict to the same value, the result is something akin to a filtration: any open set on which the assignment is consistent at a lower threshold will remain consistent at a higher threshold.  Making this intuition precise, one obtains the \emph{consistency filtration}.

 Providing precise definitions of \emph{consistency radius} and \emph{consistency filtration} occupy some of our attention in this article, but their value is squarely rooted in the properties we prove about them.  From the standpoint of applications, this article establishes that small perturbations in the data (the assignment) or the model (the sheaf) do not result in large changes in the consistency radius or in the consistency filtration.  Furthermore, the consistency filtration preserves specific relationships between related sheaves and assignments --- it is a covariant functor between appropriately constructed categories.  Although the results in this article require a limited base topological space --- sheaves on finite spaces are the focus, though a few results described in this article generalize to arbitrary topological spaces --- we do not restrict to sheaves over a single topological space.  \emph{Sheaf morphisms} will be permitted to transform a sheaf on one space to a sheaf on another space.

 This dual nature of the consistency filtration --- it is both a continuous function between topological spaces \emph{and} a covariant functor between categories --- arises because we can view the class of all sheaf-assignment pairs through two lenses: topological and categorial.  The stalks and restrictions of a given sheaf of pseudometric spaces are manifestly topological, but that same sheaf is also a special kind of functor.  This dual nature is not present for a sheaf of sets, since most of its interesting properties arise from being a functor, nor is this dual nature present for a continuous map between pseudometric spaces, since most of its interesting properties arise from being topological.  

 The topological viewpoint of sheaf-assignment pairs addresses the question of how to threshold data based on self-consistency, using the sheaf of pseudometric spaces as a tool to measure the assignment.  From this viewpoint, it is best to think of the consistency radius as a continuous function of the sheaf and of the assignment (jointly!).  Since the consistency radius is a nonnegative real number that quantifies how compatible the sheaf and assignment are (smaller means more consistent), it has a straightforward interpretation.  We can therefore turn the consistency radius on its head and ask, ``for a given real number threshold, which open sets have local consistency radius smaller than that threshold?''  The answer to this is the consistency filtration, which although it is not a filtration of subspaces in the usual sense, it is an order preserving function from the real numbers to the power set of the sheaf's base space topology, in which the ordering is given by \emph{coarsening}\footnote{The opposite of the \emph{refinement} of a cover!}.  There the topological story would seem to end, if not for the recent definition of a general \emph{interleaving distance} \cite{Harker_2018}.  With the interleaving distance in hand, this article shows that the consistency filtration is a continuous function as well.  This establishes that the consistency filtration is robust to perturbations, with \emph{persistent \v{C}ech cohomology} as a special case.

 In contrast, the categorial viewpoint emphasizes transformations of sheaf-assignment pairs, first by generalizing the notion of a sheaf morphism to address the case where the base space of the domain and the base space of the codomain differ, and secondly by focusing attention on only those sheaf morphisms that preserve \emph{both} the sheaf and assignment.   These definitions put the category of assignments of a sheaf of pseudometric spaces on firm theoretical ground.  With this specialized category in hand, consistency radius is not a functor, even though the morphisms transform it in a predictable way (Lemma \ref{lem:cr_morphism}).  This suggests that there still is a functor at work, which unsurprisingly turns out to be the consistency filtration when an appropriate category is constructed for its codomain.

 In taking these two viewpoints, this article defines several apparently novel categories for sheaf-assignment pairs and filtrations of covers, and obtains three results about their dual topological/categorial natures:
 \begin{description}
 \item[Theorem \ref{thm:consistency_radius_continuous}]: Consistency radius is a continuous function,
 \item[Theorem \ref{thm:cf_functor}]:  Consistency filtration is a covariant functor, and
 \item[Theorem \ref{thm:cf_robustness}]: Consistency filtration is a continuous function.
 \end{description}

\section{Interpretation and applications}

A \emph{sheaf of sets on a topological space} is a mathematical structure for representing local data.  While sheaves are general purpose and abstract, they do not require unrealistic assumptions about domain-specific data.  As such, they provide a convenient language for describing how systems composed of interrelated parts can interact.  In applications, this is valuable because sheaves can naturally represent systems composed of different types of subsystems.  Since subsystems are usually modeled (or engineered) separately and largely in isolation from one another, composing the state of several subsystems into a global composite state is not trivial.  If one requires that the composite system be represented as a sheaf, this imposes specific conditions on how subsystems may be assembled.  If these conditions are met, then the idea of a global composite state --- a \emph{global section} of the sheaf --- is well-defined and can be obtained algorithmically \cite{Robinson_sheafcanon}.  From a practical standpoint, this means that the process of asserting that the composite system be modeled as a sheaf simplifies and standardizes the bookkeeping tasks involved with modeling the composite system.

Over the years, these considerations have motivated various researchers to encode a number of systems as sheaves, such as
\begin{itemize}
\item Bayes nets and graphical models \cite{Robinson_multimodel},
\item Biochemical networks \cite{20161002_ACMBCB},
\item Communication networks \cite{Robinson_GlobalSIP_2014, ghrist2011network},
\item Control systems \cite{mansourbeigi2018sheaf,Mansourbeigi_2017},
\item Differential equations \cite{Robinson_qgtopo} and their discretizations (also as \emph{dual sheaves}) \cite{Robinson_multimodel},
\item Discrete and continuous dynamical systems (as \emph{cosheaves}) \cite{posettrack_techreport},
\item Flow networks \cite{nguemo2017sheaf, ghrist2013topological},
\item Formal models of interacting software objects \cite{zadrozny2018sheaf,nelaturi2016combinatorial,malcolm2009sheaves, Lilius_1993, goguen1992sheaf},
\item Multi-dimensional weighted or labeled graphs \cite{20170104_JMM_Purvine,Joslyn_GraphFest2016,Praggastis_2016},
\item Quantum information systems \cite{abramsky2015contextuality,abramsky2011sheaf},
\item Sensor networks \cite{Robinson_sheafcanon}, and
\item Signal processing chains \cite{Robinson_TSP_book}.
\end{itemize}

One may ask whether encoding any of these composite systems as a sheaves confers any particular benefit, or if sheaves are merely another universal encoding.  The payoff is that sheaf theory provides \emph{invariants}, which are general analytic tools that can be easily applied to perform certain data processing tasks, largely without imposing additional assumptions on the subsystem models.  This article discusses two of these invariants, the \emph{consistency radius} and \emph{consistency filtration}, that quantify the compatibility between the system's representation (which may be hypothetical, incomplete, or inaccurate) and data collected about the state of subsystems (which may be noisy or subject to other kinds of systematic errors).  Although this article does not directly address any of the specific systems mentioned above, the results obtained apply broadly to all of them.

\section{Historical context}

Recent years have seen a rapid growth in the field of \emph{topological data analysis}, which is the study of the topological properties of datasets.  The typical analysis involves transforming a \emph{point cloud} --- a collection of data points in a high-dimensional metric space --- into a combinatorial model of a topological space, such as an abstract simplicial complex, whose properties are then studied.  For instance, the \emph{simplicial homology} of an abstract simplicial complex can be used to identify potentially interesting features.  These combinatorial models typically depend on the choice of a threshold that is used to determine if two of the data points are ``close enough'' to be joined together.  One of the primary insights of topological data analysis is \emph{persistence}: one should not select a single threshold, but instead consider the ensemble of all thresholds.  This idea led to the \emph{persistent homology} of a point cloud \cite{Ghrist_2008,Zomorodian_2005,DeSilva_2004,Edelsbrunner_2002}, an idea that has since become quite popular because of its broad applicability and straightforward usage.

Even though persistent homology has wide appeal, it is limited to relatively homogeneous data sets, essentially those that can be represented as point clouds.  Data sets with more structure than that are difficult to study.  Sheaves on a topological space provide a good framework for studying these kind of data sets, since they permit the structure of the data to vary over the base space.  In most applications, it suffices to consider either \emph{constructible sheaves} or sheaves on a \emph{finite space} \cite{Curry_2013,Schurmann_2003}.  Traditionally, sheaves on a given topological space are studied using \emph{sheaf cohomology}, a global algebraic summary that respects the category of sheaves.  The cohomology of a sheaf includes a representation of the space of global sections of the sheaf.  Computing the cohomology is greatly simplified for constructible sheaves \cite{Shepard_1985} and can be done efficiently using discrete Morse theory \cite{curry2016discrete}.  Given this happy situation, sheaf cohomology has been useful in applications to network science \cite{nguemo2017sheaf,ghrist2011network, Robinson_qgtopo} and quantum information \cite{abramsky2015contextuality,abramsky2011sheaf}.  Sheaf cohomology contains information beyond the global sections, namely obstructions to the existence of other global sections, and this too can be useful in certain cases \cite{Robinson_ENTCS}.

Useful as it may be, sheaf cohomology cannot be defined if the stalks of a sheaf are merely sets, as it requires the stalks to have the structure of abelian groups \cite{Hubbard_2006}.  Unfortunately, in many applications, such as those listed in the previous section, this algebraic structure is simply unavailable.  However, for many systems, the stalks have geometric structure instead of algebraic structure.  Specifically, the distance between two points within a stalk can be measured with a \emph{pseudometric}.  It is therefore convenient and appropriate to consider sheaves of pseudometric spaces instead of sheaves of abelian groups.  In an earlier article \cite{Robinson_sheafcanon}, the author initiated the study of sheaves of pseudometric spaces, by showing how to encode a general model of a sensor system as a sheaf of pseudometric spaces, and how to encode the measurements made by such a system as an \emph{assignment} to that sheaf.  Under this framework, the \emph{consistency radius} can be defined, which bounds the distance between the assignment and the nearest global section of the sheaf.  This bound motivated an optimization algorithm to fuse potentially noisy or uncertain measurements made by the sensor system into a single globalized measurement, represented as a global section.

This article applies the idea of persistence --- using an ensemble of thresholds rather than one threshold --- to the consistency radius.  The result is the definition of a new mathematical object called the \emph{consistency filtration}.  The consistency filtration fits within the landscape of topological data analysis tools, since persistent \v{C}ech cohomology for point clouds is the consistency filtration for a particular sheaf and assignment (Example \ref{eg:point_clouds}), but it is considerably more general.  Indeed, the two main results proven in this article about the consistency filtration --- that it is both a covariant functor and a continuous function --- rely on a generalization of the interleaving distance \cite{Harker_2018} developed for use in persistent homology.

\section{Preliminaries}

This article involves the study of \emph{sheaves}, mathematical models of local consistency relationships between data.  Our primary focus is a \emph{sheaf of pseudometric spaces on a topological space}, which is a rather elaborate object.  It is easiest to develop the definition of this object slowly, by unpacking each term in turn, as is done in this section.

\begin{definition}
  \label{df:presheaf}
Suppose $(X,\col{T})$ is a topological space.  A \emph{presheaf $\shf{P}$ of sets on $(X,\col{T})$} consists of the following specification:
\begin{enumerate}
\item For each open set $U \in \col{T}$, a set $\shf{P}(U)$, called the \emph{stalk} at $U$,
\item For each pair of open sets $U \subseteq V$, there is a function $\shf{P}(U \subseteq V):\shf{P}(V)\to\shf{P}(U)$, called a \emph{restriction function} (or just a \emph{restriction}), such that
\item $\shf{P}(U \subseteq U)$ is the identity function and
\item For each triple $U \subseteq V \subseteq W$ of open sets, $\shf{P}(U \subseteq W) = \shf{P}(U \subseteq V) \circ \shf{P}(V \subseteq W)$.
\end{enumerate}
\end{definition}

Intuitively, a presheaf is a data structure for holding a variety of data items, localized to each open set.  The stalk $\shf{P}(U)$ at an open set $U$ specifies the possible values of a datum localized at $U$.  It is also the case that a presheaf $\shf{P}$ can be succinctly defined as ``a contravariant functor from the category of open sets of $(X,\col{T})$, with inclusion functions being morphisms, to the category $\cat{Set}$ of sets and functions.''

Recall that a \emph{finite topological space} is a topological space $(X,\col{T})$ in which $\col{T}$ is a finite set.  Finite topological spaces are automatically \emph{Alexandrov} spaces, which means that intersections of arbitrarily many open sets remains open.  For Alexandrov spaces, the \emph{star} of some subset $A \subseteq X$, given by
\begin{equation*}
  \st A := \cap\{U \in \col{T}: A \subseteq U\}
\end{equation*}
is the smallest open set that contains $A$.

\begin{remark}
Our usage of the word \emph{stalk} differs from the traditional usage, in which $\shf{P}(U)$ is called \emph{set of local sections}.  The traditional definition of a \emph{stalk} is a \emph{direct limit} of sets of the form $\shf{P}(U)$ such that each open set $U$ contains a given point $x\in X$.  The distinction between these two usages becomes primarily linguistic for finite topological spaces: the traditional definition deems $\shf{P}(U)$ a stalk at $x \in X$ whenever an open set $U$ is $\st \{x\}$, making $U$ minimal in the sense of inclusion.  If the topology is not finite, there may not necessarily be such a minimal open $U$ at each point, but the direct limit still exists.  Our usage here merely removes the minimality requirement since it happens to be unimportant in this article.
\end{remark}

The restrictions identify how to relate a datum on a large open set to one on a smaller open set.  When the smaller open set already has a value, this may or may not agree with the value propagated by a restriction function from a larger open set.  When these data agree is a special situation; the data are then called a \emph{section} as the next definition explains.

\begin{definition}
\label{df:section}
A \emph{global section} of a presheaf $\shf{P}$ of sets on a topological space $(X,\col{T})$ is an element $s$ of the direct product\footnote{Which is in general \emph{not} the direct sum, since $\col{T}$ may be infinite!} $\prod_{U \in \col{T}}\shf{P}(U)$ such that for all $U \subseteq V \in \col{T}$ then $\shf{P}(U \subseteq V)\left(s(V)\right) = s(U)$.  A \emph{local section} is defined similarly, but refers to a subcollection $\col{U}$ of $\col{T}$ instead of $\col{T}$.
\end{definition}

The set of global sections of a presheaf $\shf{P}$ may be quite different from the set $\shf{P}(X)$.  It is for this reason that an additional \emph{gluing axiom} can be included to remove this distinction.

\begin{definition}
\label{df:sheaf}
Let $\shf{S}$ be a presheaf of sets on a topological space $(X,\col{T})$.  We call $\shf{S}$ a \emph{sheaf of sets on $(X,\col{T})$} if for every open set $U \in \col{T}$ and every collection of open sets $\col{U}\subseteq \col{T}$ with $U = \cup \col{U}$, then $\shf{S}(U)$ is isomorphic to the space of local sections over the set of elements $\col{U}$.
\end{definition}

This article is concerned with more than mere agreement or disagreement; more can be said if the distance between two data can be measured.  The appropriate way to measure distance between two data is with a \emph{pseudometric}.

\begin{definition}
  \label{df:pseudometric}
  A \emph{pseudometric $d_X$} on a set $X$ is a function $d : X \times X \to \mathbb{R}$ that satisfies the following axioms:
  \begin{enumerate}
  \item $d_X(x,y) \ge 0$,
  \item $d_X(x,x) = 0$,
  \item $d_X(x,y)=d_X(y,x)$, and
  \item $d_X(x,z) \le d_X(x,y)+d_X(y,z)$ 
  \end{enumerate}
  for all $x, y, z \in X$.
  The pair $(X, d_X)$ is called a \emph{pseudometric space}.
\end{definition}

It is a standard fact that every pseudometric $d_X$ on $X$ induces a topology $\col{T}$ on $X$ generated by the \emph{open balls}
\begin{equation*}
  B_r(x) := \{ y \in X : d_X(x,y) < r\}
\end{equation*}
for each $x \in X$ and each $r > 0$.  We will usually abuse terminology slightly by regarding a pseudometric space as a topological space.
We can deem a function $f : X \to Y$ to be a \emph{continuous map} if $d_X$ is a pseudometric on $X$ and $d_Y$ is a pseudometric on $Y$, and if $f$ is a continuous function for the topologies induced by $d_X$ and $d_Y$.  We will write $f : (X,d_X) \to (Y,d_Y)$ to emphasize the situation where $f$ is a continuous map.

\begin{definition}
  \label{def:pseud}
  The category $\cat{Pseud}$ is the category of \emph{pseudometric spaces}, in which objects are pseudometric spaces and the morphisms are continuous maps.  
\end{definition}

In this article, it is often important to measure ``how continuous'' a continuous map is.  If there is a $K>0$ such that
\begin{equation*}
  d_Y(f(x),f(y)) \le K d_X(x,y)
\end{equation*}
for all $x,y \in X$, this $K$ is called a \emph{Lipschitz constant} for $f : (X,d_X) \to (Y,d_Y)$.  It is easy to demonstrate that the existence of a finite Lipschitz constant for a function $f:X \to Y$ between two pseudometric spaces $(X,d_X)$ and $(Y,d_Y)$ implies that $f$ is continuous.  Such a function is said to be \emph{Lipschitz continuous}.

With these tools in hand, namely sheaves of sets on a topological space and pseudometric spaces, we combine them into one concept that carries through the rest of the article.

\begin{definition}
  \label{def:sheaf_of_pseud}
  A \emph{sheaf $\shf{S}$ of psuedometric spaces on a topological space $(X,\col{T})$} consists of a sheaf of sets $\shf{S}$ on $(X,\col{T})$ and the choice of a pseudometric $d_U$ on $\shf{S}(U)$ for each open set $U\in \col{T}$, such that each restriction function is a continuous map.  We will call each restriction a \emph{restriction map} in this situation to emphasize that we are working with a sheaf of pseudometric spaces rather than a sheaf of sets. 
\end{definition}

Succinctly, a sheaf $\shf{S}$ of psuedometric spaces on a topological space is a kind of contravariant functor from the category of open sets in $(X,\col{T})$ to $\cat{Pseud}$.  With this interpretation in mind, we may simply think of each stalk $\shf{S}(U)$ as \emph{being} a pseudometric space $(\shf{S}(U),d_U)$, even though this may be considered a slight abuse of notation.

It is immediate from the definitions that each space of (local or global) sections of a sheaf of pseudometric spaces is itself a topological space.  In most of what follows, we will work with sheaves on a \emph{finite topological space} $(X,\col{T})$, one in which the topology $\col{T}$ is a finite set.  Each space of sections for a sheaf of pseudometric spaces on a finite topological space is then a pseudometric space as well, by the usual constructions of a pseudometric on a product space.  This idea leads to the next definition.

\begin{definition} \cite{Robinson_sheafcanon}.
  \label{def:assignment}
  For a sheaf $\shf{S}$ of sets on a topological space $(X,\col{T})$, an \emph{assignment} is an element $a \in \prod_{U\in \col{T}} \shf{S}(U)$.  If $\shf{S}$ is a sheaf of pseudometric spaces, then the set of assignments has the \emph{assignment pseudometric} given by 
  \begin{equation*}
    D(a,b) := \sup_{U \in \col{T}} d_U\left(a(U),b(U)\right).
  \end{equation*}
  For an assignment $a$ to a sheaf $\shf{S}$, each value 
  \begin{equation*}
    d_U\left(\left(\shf{S}(U \subseteq V)\right)a(V),a(U)\right)
  \end{equation*}
  where $U \subseteq V \in \col{T}$, is called a \emph{critical threshold}.  The \emph{consistency radius} given by
  \begin{equation*}
    c_{\shf{S}}(a) := \sup_{U \subseteq V \in \col{T}} d_U\left(\left(\shf{S}(U \subseteq V)\right)a(V),a(U)\right),
  \end{equation*}
  is the supremum of all critical thresholds.
\end{definition}

The central relationship between global sections of $\shf{S}$ and assignments is captured by the following bound.
\begin{proposition} \cite[Prop. 23]{Robinson_sheafcanon}
  For an assignment $a$ to a sheaf $\shf{S}$ of pseudometric spaces on $(X,\col{T})$ in which each restriction map of $\shf{S}$ is Lipschitz with constant $K$, then
  \begin{equation*}
    D(a,s) \ge \frac{c_{\shf{S}}(a)}{1+K}
  \end{equation*}
  for every global section $s$ of $\shf{S}$.
\end{proposition}

\begin{remark}
Assignments thresholded to a certain level of consistency are global sections of a (different) sheaf on a subdivision of the base space \cite{pseudosections}.  While this perspective has theoretical merit, it is useful to study the extent to which assignments are consistent to a given threshold.  This perspective is introduced by Praggastis in \cite{Praggastis_2016}, where it is shown that assignments supported on the vertices of an abstract simplicial complex yield a cover of that simplicial complex when they are thresholded to a certain level of consistency.  We take a parallel approach in this article, but instead consider assignments supported on the entire space.  The most prominent difference is in the definition of maximal consistent cover (Definition \ref{def:consistent_cover} and Lemma \ref{lem:consistent_cover}), which is analogous to the cover constructed in \cite{Praggastis_2016}, but our cover is only over a subspace and is generally finer.
\end{remark}

\begin{remark}
  \label{rem:consistency_diameter}

  The consistency radius is a \emph{radius}, rather than a diameter, because the assignment itself acts like a central point to which values propagated along the restrictions are compared.

  We could instead define a quantity
    \begin{equation*}
    d_{\shf{S}}(a) := \sup_{U \subseteq V_1 \in \col{T},} \sup_{U \subseteq V_2 \in \col{T}} d_U\left(\left(\shf{S}(U \subseteq V_1)\right)a(V_1),\shf{S}(U \subseteq V_2)a(V_2)\right),
    \end{equation*}
    which is a \emph{diameter} in that
    \begin{equation*}
      c_{\shf{S}}(a) \le d_{\shf{S}}(a) \le 2 c_{\shf{S}}(a).
    \end{equation*}
    
  The left inequality arises simply by taking $U = V_2$ since $\shf{S}(U \subseteq U) = \id_{\shf{S}(U)}$ by definition.  The right inequality is a short calculation
  \begin{eqnarray*}
    d_{\shf{S}}(a) &=& \sup_{U \subseteq V_1 \in \col{T},} \sup_{U \subseteq V_2 \in \col{T}} d_U\left(\left(\shf{S}(U \subseteq V_1)\right)a(V_1),\shf{S}(U \subseteq V_2)a(V_2)\right)\\
    &\le&\sup_{U \subseteq V_1 \in \col{T}} d_U\left(\left(\shf{S}(U \subseteq V_1)\right)a(V_1),a(U)\right) + \sup_{U \subseteq V_2 \in \col{T}} d_U\left(a(U),\shf{S}(U \subseteq V_2)a(V_2)\right)\\
    &\le& 2 c_{\shf{S}}(a).
  \end{eqnarray*}
\end{remark}

\section{Sheaves paired with assignments}

This article discusses quite a few different categories, many of which are not extensively discussed in the literature.  For convenience, all of the categories under discussion are listed below, with their descriptions (and mnemonic hints in boldface):
\begin{enumerate}
\item $\cat{Pseud}$: the category of {\bf pseud}ometric spaces, Definition \ref{def:pseud},
\item $\cat{Shv}$: the category of all {\bf sh}ea{\bf v}es, Definition \ref{def:shv},
\item $\cat{ShvFP}$: the category of all {\bf sh}ea{\bf v}es on {\bf f}inite topological spaces of {\bf p}seudometric spaces, Definition \ref{def:shv},
\item $\cat{ShvA}$: the category of all {\bf sh}ea{\bf v}es paired with {\bf a}ssignments, Definition \ref{def:shva},
\item $\cat{ShvPA}(X,\col{T})$: the category of all {\bf sh}ea{\bf v}es of {\bf p}seudometric spaces on a fixed topological space $(X,\col{T})$, each paired an {\bf a}ssignment, Definition \ref{def:shva},
\item $\cat{ShvFPA}$: the category of all {\bf sh}ea{\bf v}es on arbitrary {\bf f}inite topological spaces of {\bf p}seudometric spaces, each paired an {\bf a}ssignment, Definition \ref{def:shva},
\item $\cat{ShvFPA}_L$: the subcategory of $\cat{ShvFPA}$ whose morphisms are sheaf morphisms along base space homeomorphisms whose component maps are {\bf L}ipschitz, Theorem \ref{thm:cf_functor},

\item $\cat{PartCovers}$: the category of {\bf part}ial {\bf covers}, that consist of topological spaces with collections of open sets, Definition \ref{def:ctop},
\item $\cat{CoarseFilt}$: the category of {\bf coarse}ning {\bf filt}rations of topological spaces, Definition \ref{def:sctop}, and
\item $\cat{PMod}$: the category of {\bf p}ersistence {\bf mod}ules, Definition \ref{def:pmod}.
\end{enumerate}

\begin{definition}
  \label{def:shv}
The category $\cat{Shv}$ is the category of \emph{all sheaves}, which each object is a sheaf on some topological space.  A morphism $m$ in this category from a sheaf $\shf{P}$ on $(X,\col{T}_X)$ to a sheaf $\shf{Q}$ on $(Y,\col{T}_Y)$ consists of (1) a continuous map $f:(X,\col{T}_X) \to (Y,\col{T}_Y)$ and (2) a set of \emph{component functions} $m_U$, one for each $U\in\col{T}_Y$, such that 
\begin{equation*}
    \xymatrix{
      \shf{P}(f^{-1}(U)) \ar[r]^-{m_U} &\shf{Q}(U)\\
      \shf{P}(f^{-1}(V)) \ar[r]_-{m_V} \ar[u]^{\shf{P}(f^{-1}(U) \subseteq f^{-1}(V))}&\shf{Q}(V) \ar[u]_{\shf{Q}(U \subseteq V)}\\
      }
\end{equation*}
commutes for each pair of open sets $U \subseteq V$ in $\col{T}_Y$.

Composition of morphisms in $\cat{Shv}$ is defined by composing both the continuous maps and the component functions.  Explicitly, if $n:\shf{Q}\to\shf{R}$ is a morphism from the sheaf $\shf{Q}$ defined above to a sheaf $\shf{R}$ on $(Z,\col{T}_Z)$ along the continuous map $g: (Y,\col{T}_Y) \to (Z, \col{T}_Z)$ with component functions $n_U$ for each $U \in \col{T}_Z$, then the composition $n \circ m$ is given by the following commutative diagram
\begin{equation*}
      \xymatrix{
      \shf{P}((g \circ f)^{-1}(U))) \ar[rr]^-{m_{g^{-1}(U)}}\ar@/^2pc/[rrrr]^{(n\circ m)_U} &&\shf{Q}(g^{-1}(U)) \ar[rr]^-{n_U} && \shf{R}(U)\\
      \shf{P}((g \circ f)^{-1}(V))) \ar[rr]_-{m_{g^{-1}(V)}} \ar@/_2pc/[rrrr]_{(n\circ m)_V}\ar[u]^{\shf{P}((g \circ f)^{-1}(U)) \subseteq (g \circ f)^{-1}(V)))}&&\shf{Q}(g^{-1}(V)) \ar[u]_{\shf{Q}(g^{-1}(U) \subseteq g^{-1}(V))} \ar[rr]_-{n_V} && \shf{R}(V) \ar[u]_{\shf{R}(U \subseteq V)} \\
      }
\end{equation*}

  The subcategory $\cat{ShvFP}$ is the category of sheaves of pseudometric spaces on finite topological spaces.  Because of Definition \ref{def:sheaf_of_pseud}, a sheaf of pseudometric spaces always has continuous restriction maps.  We will additionally assume that the component functions of every morphism of $\cat{ShvFP}$ are also continuous, and will emphasize this by calling them \emph{component maps}.  Under this convention, $\cat{ShvFP}$ is not a full subcategory of $\cat{Shv}$.
\end{definition}

\begin{remark}
  The category $\cat{Shv}$ is a generalization of the more traditional $\cat{Shv}(X,\col{T})$, which is the category of sheaves on a fixed topological space $(X,\col{T})$ in which each object is a sheaf on $(X,\col{T})$ and all morphisms are along the identity map $X\to X$.
\end{remark}

\begin{definition}
  \label{def:shva}
  Each object of the category $\cat{ShvA}$ of \emph{sheaf assignments} consists of a sheaf $\shf{S}$ on $(X,\col{T})$ and \emph{assignment} $a \in \prod_{U\in\col{T}} \shf{S}(U)$ to that sheaf.  In this category, morphisms consist of sheaf morphisms that are compatible with the assignment, in the following way.  Suppose that $a$ is an assignment to a sheaf $\shf{S}$ on a space $X$ and $b$ is an assignment to a sheaf $\shf{R}$ on a space $Y$.  A sheaf morphism $m: \shf{S} \to \shf{R}$ along a continuous map $f:X\to Y$ is morphism $(\shf{S},a) \to (\shf{R},b)$ in $\cat{ShvA}$ if for each open $U\subseteq Y$, the component function
  \begin{equation*}
    m_U : \shf{S}(f^{-1}(U)) \to \shf{R}(U)
  \end{equation*}
  satisfies
  \begin{equation*}
    m_U(a(f^{-1}(U))) = b(U).
  \end{equation*}

  We will also make use of the following subcategories:
  \begin{itemize}
  \item The subcategory $\cat{ShvFPA}$ of all sheaves on finite topological spaces of pseudometric spaces, each paired an assignment, and
  \item The subcategory $\cat{ShvPA}(X,\col{T})$ of all sheaves of pseudometric spaces on a fixed topological space $(X,\col{T})$, each paired an assignment.
  \end{itemize}
\end{definition}

\begin{proposition}
  There is a forgetful functor $\cat{ShvFPA}\to\cat{ShvFP}$, given by merely forgetting the assignment.
\end{proposition}
\begin{proof}
  Observe that without the assignments, the morphisms of $\cat{ShvFPA}$ are merely sheaf morphisms.
\end{proof}

\begin{proposition} 
  There is a functor $\cat{ShvFP}\to\cat{Pseud}$, given by replacing a sheaf by its space of assignments using the assignment pseudometric given in Definition \ref{def:assignment}.
\end{proposition}
This proposition is not true for sheaves over infinite spaces.  Perhaps surprisingly, requiring compactness of the base spaces is not sufficient to resolve the difficulty because assignments can lack all internal consistency.
\begin{proof}
  Sheaf morphisms were defined to support this!  Specifically, suppose that $m:\shf{S} \to \shf{R}$ is a morphism of sheaves of pseudometric spaces along the continuous map $f:(X,\col{T}_X) \to (Y,\col{T}_Y)$ between topological spaces, so that it is a morphism in $\cat{ShvFP}$.  Define a function $Fm: \prod_{U \in \col{T}_X} \shf{S}(U) \to \prod_{V \in \col{T}_Y} \shf{R}(V)$ between spaces of assignments for the two sheaves by its action on an arbitrary assignment $a$ to $\shf{S}$
  \begin{equation*}
    \left(Fm(a)\right)(V) = m_V(a(f^{-1}(V))),
  \end{equation*}
  for each $V \in \col{T}_Y$, which should be read as defining an assignment to $\shf{R}$.
  \begin{description}
  \item[$Fm$ is continuous.] Let $\epsilon > 0$ and an assignment $a\in\prod_{U \in \col{T}_X} \shf{S}(U)$ be given.
    Observe that since each component map $m_V$ of the sheaf morphism is continuous according to the convention established in Definition \ref{def:shv}
    for each open $V \in \col{T}_Y$, there is a $\delta_V > 0$ such that for every $x \in \shf{S}(f^{-1}(V))$ with
    \begin{equation*}
      d_{f^{-1}(V)}(x,a(f^{-1}(V))) < \delta_V,
    \end{equation*}
    it follows that
    \begin{equation*}
      d_V(m_V(x),Fm(a)_V) < \epsilon.
    \end{equation*}
    Since objects in $\cat{ShvFP}$ consist of sheaves over finite spaces, the infimum
    \begin{equation*}
      \delta := \inf_{V \in \col{T}_Y} \delta_V
    \end{equation*}
    is strictly positive.  (If $\col{T}_Y$ is infinite, this may be false.) Thus if $b$ is an assignment to $\shf{S}$ with $D(a,b) < \delta$, then $D(Fm(a),Fm(b)) < \epsilon$.  As an aside, observe that $D(a,b) < \delta$ is stronger than necessary, since this contrains all stalks of $\shf{S}$ not just those that are preimages of $\col{T}_Y$-open sets.
    
  \item[This defines a functor.] Let $n:\shf{R} \to \shf{P}$ be another morphism in $\cat{ShvFP}$, this one along a continuous map $g:(Y,\col{T}_Y) \to (Z,\col{T}_Z)$.  This implies that for an open set $V \in \col{T}_Z$, $g^{-1}(V)$ and $(g \circ f)^{-1}(V)$ are also both open.  We then compute
    \begin{eqnarray*}
      \left(F(n\circ m)(a)\right)(V) &=& (n\circ m)_V(a((g\circ f)^{-1}(V)))\\
      &=& (n\circ m)_V(a(f^{-1}\circ g^{-1}(V)))\\
      &=& (n\circ m)_V(a(f^{-1} ( g^{-1}(V))))\\
      &=& n_{V} ( m_{g^{-1}V}(a(f^{-1} ( g^{-1}(V)))))\\
      &=& n_{V} ( (Fm(a))(g^{-1}(V)) ) )\\
      &=& \left(Fn (Fm(a))\right)(V).
    \end{eqnarray*}
  \end{description}
\end{proof}

\begin{proposition}
  \label{prop:shva_sections}
  In $\cat{ShvA}$, any morphism whose domain $(\shf{S},s)$ is a sheaf $\shf{S}$ and a global section $s$ of $\shf{S}$ will have a global section as its codomain.
\end{proposition}
\begin{proof}
  This follows immediately from the fact that images of global sections through sheaf morphisms are global sections.
\end{proof}

\begin{figure}
  \begin{center}
    \includegraphics[width=2in]{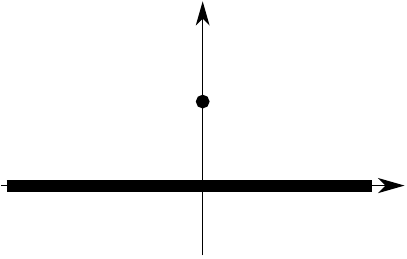}
    \caption{The subspace $(X,d)$ of $\mathbb{R}^2$ in Example \ref{eg:homeo_intransitive}}
    \label{fig:homeo_intransitive}
  \end{center}
\end{figure}

Global sections are sometimes --- but not always --- isomorphic objects in $\cat{ShvFPA}$ as the following example shows.
\begin{example}
  \label{eg:homeo_intransitive}
  Consider the metric space $(X,d)$ that is a subspace of $\mathbb{R}^2$ with the usual metric $d$, with
  \begin{equation*}
    X := \{(x,0) : x \in \mathbb{R}\} \cup \{(0,1)\}
  \end{equation*}
  as shown in Figure \ref{fig:homeo_intransitive}.  The action of the group of homeomorphisms on $(X,d)$ is not transitive because every homeomorphism of $(X,d)$ must leave the point $(0,1)$ fixed.

  Now consider the sheaf $\shf{S}$ on the single-point space $\{*\}$ (with the only possible topology for that space) with $\shf{S}(\{*\}) = (X,d)$.  Both the space of global sections of $\shf{S}$ and the space of assignments of $\shf{S}$ are therefore precisely the points of $(X,d)$.  Any isomorphism $\shf{S}\to\shf{R}$ in $\cat{Shv}$ or $\cat{ShvPA}$ will induce a homeomorphism on $(X,d)$.  Thus $(\shf{S},a)$ and $(\shf{R},b)$ are isomorphic if and only if $a$ and $b$ are elements of the same connected component of $(X,d)$.
\end{example}

\begin{proposition}
  \label{prop:shvpa_topology}
For any topological space $(X,\col{T})$, each isomorphism class of $\cat{ShvPA}(X,\col{T})$ can be given a topology that is a subspace of the product of the assignment pseudometric and copies of the topology of uniform convergence (one copy for each restriction).
\end{proposition}

Recall that the \emph{topology of uniform convergence} for two pseudeometric spaces $(X,d_X)$, $(Y,d_Y)$ provides a pseudometric $D$ on $C(X,Y)$ given by
\begin{equation*}
  D(f,g) := \sup_{x\in X} d_Y(f(x),g(x)).
\end{equation*}

\begin{proof}
  Fix an isomorphism class of $\cat{ShvPA}(X,\col{T})$ by specifying a representative sheaf $\shf{S}$ of pseudometric spaces and a representative assignment $a$ to $\shf{S}$.  Every sheaf isomorphic to $\shf{S}$ has stalks that are isomorphic to those of $\shf{S}$ but possibly different restriction maps.  Therefore, the space of assignments for every sheaf isomorphic to $\shf{S}$ is the same.  Given these two facts, the sheaves isomorphic to $\shf{S}$ are parameterized by appropriate choices of restriction maps from the product
  \begin{equation*}
    \prod_{U \subseteq V \in \col{T}} C(\shf{S}(V),\shf{S}(U)),
  \end{equation*}
  in which each factor is given the topology of uniform convergence.

  Thus the isomorphism class of $(\shf{S},a)$ in $\cat{ShvPA}(X,\col{T})$ is merely a subset of the product
  \begin{equation*}
    \prod_{U \subseteq V \in \col{T}} C(\shf{S}(V),\shf{S}(U)) \times \prod_{U \in \col{T}} \shf{S}(U),
  \end{equation*}
  in which the second factor in the product has the topology induced by the assignment pseudometric.
\end{proof}

\begin{corollary}
  \label{cor:shvpa_is_a_space}
  For any topological space $(X,\col{T})$, $\cat{ShvPA}(X,\col{T})$ can be made into topological space by simply taking the disjoint union of all the isomorphism classes.
\end{corollary}

Observe that each isomorphism class is $\cat{ShvPA}(X,\col{T})$ is actually a closed subspace of the product
  \begin{equation*}
    \prod_{U \subseteq V \in \col{T}} C(\shf{S}(V),\shf{S}(U)) \times \prod_{U \in \col{T}} \shf{S}(U),
  \end{equation*}
because the commutativity and gluing axioms will prohibit many choices of restriction maps.  This means that the topology of $\cat{ShvPA}(X,\col{T})$ will likely be quite complicated.

\begin{theorem}
  \label{thm:consistency_radius_continuous}
If $(X,\col{T})$ is a finite topological space, then consistency radius is a continuous function $\cat{ShvPA}(X,\col{T})\to \mathbb{R}$.
\end{theorem}

This function is continuous in the assignment metric for assignments to a single sheaf \cite{Robinson_sheafcanon}.  But it is also continuous in the compact open topology for sheaves for a fixed assignment.  Finiteness of $\col{T}$ is not essential, but makes for considerably less delicate argumentation.  

\begin{proof}
  Without loss of generality, we restrict attention to an isomorphism class of $\cat{ShvPA}(X,\col{T})$.  From the proof of Proposition \ref{prop:shvpa_topology}, the elements of this isomorphism class are a subspace of
  \begin{equation*}
    \prod_{U \subseteq V \in \col{T}} C(\shf{S}(V),\shf{S}(U)) \times \prod_{U \in \col{T}} \shf{S}(U).
  \end{equation*}

  Before we proceed with the main argument, suppose that $a$ is an assignment to a sheaf $\shf{S}$, that $\shf{R}$ is a sheaf isomorphic to $\shf{S}$, and that $b$ is an assignment to $\shf{R}$.  Recalling that
  \begin{equation*}
    \left| \| x \| - \|  y \| \right| \le \| x-y \|
  \end{equation*}
  for any norm $\| \cdot \|$, we have that
  \begin{eqnarray*}
    | c_{\shf{R}}(b) - c_{\shf{S}}(a) | & = & \left | \sup_{U_1 \subseteq V_1 \in \col{T}} d_{U_1}\left(\left(\shf{R}(U_1 \subseteq V_1)\right)b(V_1),b(U_1)\right) \right. \\&& \left.- \sup_{U_2 \subseteq V_2 \in \col{T}} d_{U_2}\left(\left(\shf{S}(U_2 \subseteq V_2)\right)a(V_2),a(U_2)\right) \right|\\
    & \le & \sup_{U \subseteq V \in \col{T}} \left | d_{U}\left(\left(\shf{R}(U \subseteq V)\right)b(V),b(U)\right) - d_{U}\left(\left(\shf{S}(U \subseteq V)\right)a(V),a(U)\right) \right|\\
    & \le & \sup_{U \subseteq V \in \col{T}} \left | d_{U}\left(\left(\shf{R}(U \subseteq V)\right)b(V),b(U)\right) - d_{U}\left(\left(\shf{S}(U \subseteq V)\right)a(V),a(U)\right) \right.\\
    &&+d_U\left(\shf{S}(U \subseteq V)b(V),b(U)\right)-d_U\left(\shf{S}(U \subseteq V)b(V),b(U)\right)\\
    &&\left.+d_U\left(\shf{S}(U \subseteq V)a(V),b(U)\right)-d_U\left(\shf{S}(U \subseteq V)a(V),b(U)\right)\right|\\
    &\le& \sup_{U \in \col{T}} d_U(a(U),b(U)) \\&& +  \sup_{U \subseteq V \in \col{T}} d_U(\shf{S}(U \subseteq V)b(V), \shf{S}(U \subseteq V)a(V))\\&& + \sup_{U \subseteq V \in \col{T}} d_U(\shf{R}(U \subseteq V)b(V), \shf{S}(U \subseteq V) b(V).\\
  \end{eqnarray*}
  
  Returning to the task at hand, suppose that $a$ is an assignment to a sheaf $\shf{S}$ and $\epsilon>0$ is given.  We show that there is an open set $Q \subseteq \cat{ShvPA}(X,\col{T})$ containing $(\shf{S},a)$ whose consistency radii all lie in the interval
  \begin{equation*}
    (c_{\shf{S}}(a) - \epsilon, c_{\shf{S}}(a) + \epsilon) \subseteq \mathbb{R}.
  \end{equation*}
  To this end, we construct $Q$ as the intersection of three open subsets $Q_1$, $Q_2$, $Q_3$ of $\cat{ShvPA}(X,\col{T})$.  The set $Q_i$ consists of the $(\shf{R},b)$ for which the $i$-th term in the last expression above is bounded by $\epsilon/3$, as follows:

\begin{itemize}
\item The first term merely requires that the distance from $a$ to $b$ in the assignment metric be not more than $\epsilon/3$; hence $b$ lies in an open subset of the space of assignments
  \begin{equation*}
    Q_1 := \{(\shf{R},b) : D(a,b) < \epsilon/3 \}.
  \end{equation*}
  
\item Since the restriction maps of $\shf{S}$ are continuous, bounding the second term by $\epsilon/3$ requires that $b$ lie within the intersection of finitely many open sets in the space of assignments
\begin{equation*}
    Q_2 := \left\{(\shf{R},b) : b(V) \in \left(\shf{S}(U \subseteq V)\right)^{-1}\left(B_{\epsilon/3}(\shf{S}(U \subseteq V)a(V))\right) \text{ for all } U\subseteq V \in \col{T} \right\},
\end{equation*}
where $B_{\epsilon/3}(x)$ is the open ball of radius $\epsilon/3$ centered at $x$.  (If the topology $\col{T}$ is not finite, this can fail since there may be infinitely open subsets $U$ of a given open set $V$.  All of these need to be intersected in the space of assignments for the factor corresponding to the stalk $\shf{S}(V)$.)
  
\item Bounding the last term above by $\epsilon/3$ requires that $\shf{R}(U\subseteq V)$ lies in an open set in the topology of uniform convergence
  \begin{equation*}
    Q_3 := \left\{(\shf{R},b) : \shf{R}(U\subseteq V) \in B_{\epsilon/3}\left(\shf{S}(U \subseteq V)\right) \text{ for all } U\subseteq V \in \col{T}\right\}.
  \end{equation*}
\end{itemize}

Thus, if $(\shf{R},b) \in Q = Q_1 \cap Q_2 \cap Q_3$, then $|c_{\shf{S}}(a) - c_{\shf{R}}(b)| < \epsilon$, which completes the argument.

\end{proof}

\begin{remark}
  Assuming the topology $\col{T}$ on a space $(X,\col{T})$ is finite, consistency radius factors into two continuous maps,
  \begin{equation*}
    \xymatrix{
      \cat{ShvPA}(X,\col{T}) \ar[r]^(0.7){F} & \mathbb{R}^n \ar[r]^N & \mathbb{R} 
    }
  \end{equation*}
  in which $n$ is the number of open set inclusions $U \subseteq V$ in the topology $\col{T}$, the second map is the infinity norm $N(v) = \|v\|_\infty$ and the first is given by
  \begin{equation*}
    F(\shf{S},a) := \left(d_U( \shf{S}(U \subseteq V) a(V), a(U)) \right)_{U \subseteq V \in \col{T}}.
  \end{equation*}
  
  Because of the equivalence of topologies of all finite-dimensional normed spaces, consistency radius is still continuous if we instead were to define the consistency radius using another norm as the second map $N$.  Of particular interest is the $2$-norm, resulting in a different definition for the consistency radius
  \begin{equation}
    \label{eq:consistency_radius_L2}
    c_{\shf{S}}(a) := \sqrt{\sum_{U \subseteq V \in \col{T}} d_U\left(\left(\shf{S}(U \subseteq V)\right)a(V),a(U)\right)^2}.
  \end{equation}
  Proving bounds (such as appear in the proofs of Theorem \ref{thm:consistency_radius_continuous}, Lemma \ref{lem:cr_morphism}, and Theorem \ref{thm:cf_robustness}) with this definition is somewhat more cumbersome, but it preserves the smooth structure if the sheaves take values in the category of Riemannian manifolds.
\end{remark}

 \section{Coarsening filtrations}

 In this section, we establish properties collections of nested collections of open sets in a topological space.  The most prominent of these is the \emph{consistency filtration} defined in Definition \ref{def:consistency_filtration}.

 \begin{lemma} (Standard)
   \label{lem:cech_refinement}
   If $\col{V}$ and $\col{U}$ are open covers of a topological space $(X,\col{T})$, when $\col{V}$ refines $\col{U}$ this induces a homomorphism $\check{H}^\bullet(X;\col{U}) \to \check{H}^\bullet(X;\col{V})$ on \v{C}ech cohomology.
 \end{lemma}

 It is worth noting that the homomorphism --- as usually defined (see for instance \cite{Hubbard_2006}) --- uses a \emph{refinement function} $\tau: \col{V} \to \col{U}$ for which $V \subseteq \tau(V)$ for all $V \in \col{V}$.  The homomorphism on \v{C}ech cohomology apparently depends on such a $\tau$, but in fact the homomorphism is indepedent of the refinement function.

 \begin{definition}
   \label{def:ctop}
   Each object of the category $\cat{PartCovers}$ of \emph{partially covered topological spaces} is a triple $(X,\col{T},\col{U})$, where $(X,\col{T})$ is a topological space and $\col{U} \subseteq \col{T}$ is a collection\footnote{$\col{U}$ may not be a cover of $X$, but covers a subspace.  This is important for Definition \ref{def:consistent_cover}.} of open sets.  Each morphism $(X,\col{T}_X,\col{V}) \to (X,\col{T}_Y,\col{U})$ of $\cat{PartCovers}$ is given by a continuous map $f:(X,\col{T}_X) \to (Y,\col{T}_Y)$ such that each $V \in \col{V}$ is a subset of $f^{-1}(U)$ for some $U\in \col{U}$.  We will sometimes abuse notation and write $f: (X,\col{T}_X,\col{V}) \to (X,\col{T}_Y,\col{U})$ in this case.  Composition of morphisms in $\cat{PartCovers}$ is given by composition of the underlying continuous maps.
   \end{definition}

 \begin{lemma}
   \label{lem:cech_morphism}
   Suppose that $\col{V}$ is a collection of open sets in a topological space $(X,\col{T})$ and $\col{U}$ is a collection of open sets in a topological space $(Y,\col{S})$.  A continuous map $f:(X,\col{T})\to (Y,\col{S})$ induces a homomorphism $\check{H}^\bullet(Y\cap\bigcup\col{U};\col{U})\to \check{H}^\bullet(X\cap\bigcup\col{V};\col{V})$ if for all $V \in \col{V}$, there is a $U \in \col{U}$ such that $V \subseteq f^{-1}(U)$.
 \end{lemma}

 Thus, \v{C}ech cohomology is a contravariant functor from $\cat{PartCovers}$ to the category of abelian groups.
 
 \begin{proof}
   Since $\col{U}$ contains open sets and $f$ is continuous, then
   \begin{equation*}
     f^{-1}(\col{U}) = \{ f^{-1}(U) : U \in \col{U} \}
   \end{equation*}
   is a collection of open sets in $X$.  Thus there is a chain map $\check{C}^k(Y\cap\bigcup\col{U};\col{U}) \to \check{C}^k(X\cap\bigcup f^{-1}(\col{U});f^{-1}(\col{U}))$ for every $k$.  Observe that $\col{V}$ refines $f^{-1}(\col{U})$ by assumption, so Lemma \ref{lem:cech_refinement} means that the desired homomorphism is the composition
   \begin{equation*}
     \check{H}^\bullet(Y\cap\bigcup\col{U};\col{U}) \to \check{H}^\bullet(X\cap\bigcup f^{-1}(\col{U});f^{-1}(\col{U})) \to \check{H}^\bullet(X\cap\bigcup\col{V};\col{V}).
   \end{equation*}
 \end{proof}

 \begin{definition}
   \label{def:sctop}
   The category of \emph{coarsening filtrations} $\cat{CoarseFilt}$ describes a collection of refinements of partial covers on fixed topological spaces.  Each object is a triple $(X,\col{T},{\bf V})$ where $(X,\col{T})$ is a topological space and ${\bf V}$ is a function
   \begin{equation*}
     {\bf V}:\mathbb{R} \to 2^{\col{T}},
   \end{equation*}
   for which ${\bf V}(t)$ is a collection of open sets in $(X,\col{T})$ for all $t$ and ${\bf V}(t)$ refines ${\bf V}(t')$ whenever $t \le t'$.  If $(X,\col{T},{\bf V})$ and $(Y,\col{S},{\bf U})$ are objects of $\cat{CoarseFilt}$, a morphism $(X,\col{T},{\bf V})\to(Y,\col{S},{\bf U})$ in $\cat{CoarseFilt}$ consists of an order preserving function $\phi:\mathbb{R}\to\mathbb{R}$ and a family of continuous functions $f_t:(X,\col{T})\to(Y,\col{S})$ such that for all $t$, ${\bf V}(s)$ refines $f_t^{-1}({\bf U}(t))$ for all $s \in \phi^{-1}(t)$.

   Composition of morphisms in $\cat{CoarseFilt}$ is given by composing the order-preserving functions and the continuous maps.  Explicitly, if $\phi_1,f_\bullet: (X,\col{T},{\bf V})\to(Y,\col{S},{\bf U})$ and $\phi_2,g_\bullet: (Y,\col{S},{\bf U})\to(Z,\col{R},{\bf W})$ are two morphisms in $\cat{CoarseFilt}$, then their composition is given by $\phi = \phi_2 \circ \phi_1$ (still order preserving) and $h_t = g_t \circ f_{\sup \phi_2^{-1}(t)}$. 
 \end{definition}

 To see that composition in $\cat{CoarseFilt}$ is well-defined, observe that if $r \in \phi^{-1}(t) = (\phi_2 \circ \phi_1)^{-1}(t)$, then $\phi_1(r) \in \phi_2^{-1}(t)$.  Therefore, ${\bf U}(\phi_1(r))$ refines $g^{-1}_t({\bf W}(t))$.  Since ${\bf V}(r)$ refines\\ $f_{\phi_1(r)}^{-1}({\bf U}(\phi_1(r)))$, these two refinements together mean that ${\bf V}(r)$ refines $f_{\phi_1(r)}^{-1}\left(g^{-1}_t({\bf W}(t))\right)$, which itself refines $h_t^{-1}({\bf W}(t))$ because $\phi_1(r) \le \sup \phi_2^{-1}(t)$.

 \begin{proposition}
   There is a fully faithful functor taking $\cat{CoarseFilt}$ into the category of sheaves of $\cat{PartCovers}$ on the Alexandrov space $(\mathbb{R},\le)$.
 \end{proposition}
 \begin{proof}
   Given an object $(X,\col{T},{\bf V})$ of $\cat{CoarseFilt}$, define a sheaf $\shf{S}$ by its stalks
   \begin{equation*}
     \shf{S}\left([t,\infty)\right) := (X,\col{T},{\bf V}(t))
   \end{equation*}
   for each $t \in \mathbb{R}$ (the smallest open set containing $t$ in the Alexandrov topology for $(\mathbb{R},\le)$ is $[t,\infty)$) and restrictions
   \begin{equation*}
     \shf{S}\left([t',\infty) \subseteq [t,\infty)\right) := \id_{X}
   \end{equation*}
   for each $t \le t'$.  The restrictions are morphisms in $\cat{PartCovers}$ because ${\bf V}(t)$ refines ${\bf V}(t')$.  Evidently this assignment is one-to-one on objects.

   Each $\cat{CoarseFilt}$ morphism $(X,\col{T},{\bf V})\to(Y,\col{S},{\bf U})$ consisting of an order preserving function $\phi:\mathbb{R}\to\mathbb{R}$ and a family of continuous functions $f_t:(X,\col{T})\to(Y,\col{S})$ is transformed into a sheaf morphism according to the following procedure.  Suppose that $t \le t'$.  Since order preserving maps are continuous in the Alexandrov topology, this means that $\phi^{-1}([t,\infty))$ is an Alexandrov-open set of the form
     \begin{equation*}
       \phi^{-1}([t,\infty)) = [\inf \phi^{-1}(t),\infty).
     \end{equation*}
     Additionally, if $t \le t'$ this implies that $\inf \phi^{-1}(t) \le \inf \phi^{-1}(t')$.  Thus any morphism between any two such sheaves along an order preserving map $\phi:\mathbb{R} \to \mathbb{R}$ consists of a commutative diagram
   \begin{equation*}
     \xymatrix{
     (X,\col{T}_X,{\bf V}(\inf \phi^{-1}(t))) \ar[d]_{\id_X} \ar[r]^(0.6){m_t} & (Y,\col{T}_Y,{\bf U}(t)) \ar[d]^{\id_Y}\\
       (X,\col{T}_X,{\bf V}(\inf \phi^{-1}(t'))) \ar[r]^(0.6){m_{t'}} & (Y,\col{T}_Y,{\bf U}(t'))
       }
   \end{equation*}
   where $m_t$ and $m_{t'}$ are morphisms in $\cat{PartCovers}$.  Morphisms in $\cat{PartCovers}$ consist of continuous maps $f_t,f_{t'}: (X,\col{T}_X) \to (Y,\col{T}_Y)$ such that
   \begin{itemize}
   \item For every $V \in {\bf V}(\inf \phi^{-1}(t))$, there is a $U \in {\bf U}(t)$ such that $V \subseteq f_t^{-1}(U)$ and
   \item For every $V' \in {\bf V}(\inf \phi^{-1}(t'))$, there is a $U' \in {\bf U}(t')$ such that $V' \subseteq f_{t'}^{-1}(U')$.
   \end{itemize}
   This is precisely --- not more nor less than --- what is defined by a morphism in $\cat{CoarseFilt}$, so the functor we are constructing must be full and faithful.

   Composition of morphisms in $\cat{CoarseFilt}$ is preserved by their transformation to sheaf morphisms as follows: given a pair of $\cat{CoarseFilt}$ morphisms
   \begin{equation*}
     \xymatrix{
       (X,\col{T}_X,{\bf V}) \ar[r]^-{\phi_1,f_\bullet} &  (Y,\col{T}_Y,{\bf U})\ar[r]^-{\phi_2,g_\bullet} & (Z,\col{T}_Z,{\bf W}),
       }
   \end{equation*}
   we have just constructed the following pair of morphisms in $\cat{PartCovers}$
   \begin{equation*}
     \xymatrix{
       (X,\col{T}_X,{\bf V}(\inf (\phi_1^{-1}\circ \phi_2^{-1})(t))) \ar[rr]^-{m_{1,\inf\phi_2^{-1}(t)}} && (Y,\col{T}_Y,{\bf U}(\inf \phi_2^{-1}(t)))  \ar[r]^-{m_{2,t}} &  (Z,\col{T}_Z,{\bf W}(t)) 
     }
   \end{equation*}
   for each $t$.  On the other hand, the composition of the $\cat{CoarseFilt}$ morphisms is given by the order preserving function $\phi = \phi_2 \circ \phi_1$ and the family of continuous functions $h_t = g_t \circ f_{\sup \phi_2^{-1} (t)}$.  We need to show that our construction produces the same composition in $\cat{PartCovers}$ as the composition from $\cat{CoarseFilt}$.
 
   If $r \in (\phi_1^{-1} \circ \phi_2^{-1})(t)$, the well-definedness of composition in $\cat{CoarseFilt}$ requires that ${\bf V}(r)$ refines $h_t^{-1}({\bf W}(t))$.  Since $\inf (\phi_1^{-1} \circ \phi_2^{-1})(t) \le r$, it follows that ${\bf V}(\inf (\phi_1^{-1} \circ \phi_2^{-1})(t))$ refines $h_t^{-1}({\bf W}(t))$ as well.  More plainly, we have shown that for every $V \in {\bf V}(\inf (\phi_1^{-1}\circ \phi_2^{-1})(t))$, there is a $W \in {\bf W}(t)$ such that $V \subseteq h_t^{-1}(W)$, namely that this composition in $\cat{CoarseFilt}$ yields a composition in $\cat{PartCovers}$ for each $t$.
  
 \end{proof}

 Note that our definition of the morphisms in $\cat{CoarseFilt}$ is consistent with the definition in \cite[Def. 2.12]{Harker_2018} of a \emph{$\phi$-shifted persistence module morphism}.

  \begin{definition} (equivalent to \cite[Def. 2.12]{Harker_2018})
    \label{def:pmod}
    The category $\cat{PMod}$ of \emph{shifted persistence modules} has persistence modules as its objects and persistence module morphisms shifted by order preserving functions as its morphisms.  Each object of $\cat{PMod}$ is a functor from the poset category $(\mathbb{R},\le)$ to the category of vector spaces.  Explicitly, each object is a choice of a vector space ${\bf E}(t)$ for each real number $t$ and a choice of linear map $f_{{\bf E},t\le t'} : E(t) \to E(t')$ for each $t \le t'$ such that
    \begin{equation*}
      f_{{\bf E}, t \le t''} = f_{{\bf E},t' \le t''} \circ  f_{{\bf E},t \le t'}
    \end{equation*}
    and
    \begin{equation*}
      f_{{\bf E},t \le t} = \id_{{\bf E}(t)}.
    \end{equation*}

    Each morphism ${\bf E} \to {\bf F}$ of $\cat{PMod}$ is an order preserving map $\phi: \mathbb{R} \to \mathbb{R}$ and a choice of a set of maps $g_t : {\bf E}(\inf \phi^{-1}(t)) \to {\bf F}(t)$ such that the following diagram commutes
    \begin{equation*}
      \xymatrix{
        {\bf E}(\inf \phi^{-1}(t)) \ar[rrr]^{f_{{\bf E},\inf \phi^{-1}(t) \le \inf \phi^{-1}(t')}} \ar[d]_{g_t} &&& {\bf E}(\inf \phi^{-1}(t')) \ar[d]^{g_t'} \\
        {\bf F}(t) \ar[rrr]_{f_{{\bf F},t\le t'}} &&& {\bf F}(t')\\
      }
    \end{equation*}
    for all $t \le t'$.
  \end{definition}

 Given this algebraic structure, a compatible geometric structure is given by \emph{interleaving}.  For our purposes, we use a definition that is consistent with the definition of $\delta$-matchings in \cite{Bauer_2014} and $(\tau,\sigma)$-interleavings in \cite[Def. 2.13]{Harker_2018}.

 \begin{definition}
   \label{def:interleaving}
    A pair of morphisms $f_t:(X,\col{T},{\bf V}) \to (Y,\col{S},{\bf U})$, $g_t:(Y,\col{S},{\bf U}) \to (X,\col{T},{\bf V})$ along $\phi$, $\psi$ (respectively) in $\cat{CoarseFilt}$ is \emph{an $\epsilon$-interleaving} if each of the following hold for all $t$:
    \begin{enumerate}
    \item $|\phi(t) - t| < \epsilon$,
    \item $|\psi(t)-t| < \epsilon$,
    \item ${\bf V}(\inf (\psi \circ \phi)^{-1}(t))$ refines ${\bf V}(t)$, and
    \item ${\bf U}(\inf (\phi \circ \psi)^{-1}(t))$ refines ${\bf U}(t)$.
    \end{enumerate}

    An \emph{$\epsilon$-interleaving} between two morphisms $f_t:{\bf E} \to {\bf F}$, $g_t:{\bf E} \to {\bf F}$ along $\phi$, $\psi$ (respectively) in $\cat{PMod}$ may be defined in exactly the same way.
  \end{definition}

  \begin{definition} (compare the definition in \cite{Bauer_2014})
    For each pair of objects $(X,\col{T},{\bf V})$, $(Y,\col{S},{\bf U})$ in $\cat{CoarseFilt}$, the function
    \begin{equation*}
      D((X,\col{T},{\bf V}), (Y,\col{S},{\bf U})) := \inf \{\epsilon : \text{there is an }\epsilon \text{-interleaving }(X,\col{T},{\bf V}) \to (Y,\col{S},{\bf U}) \}
    \end{equation*}
    is a pseudometric, called the \emph{interleaving distance}.

    The same definition works for objects in $\cat{PMod}$, \emph{mutatis mutandis}.
  \end{definition}

  \begin{proposition}
    \label{lem:cech_functor}
   Persistent \v{C}ech cohomology is a contravariant functor $PH:\cat{CoarseFilt}\to\cat{PMod}$.  
 \end{proposition}
 \begin{proof}
   If $t < t'$, a morphism $(X,\col{T},{\bf V})\to(Y,\col{S},{\bf U})$ in $\cat{CoarseFilt}$ along an order preserving $\phi$ and a family of continuous functions $f_t:(X,\col{T})\to(Y,\col{S})$ induces the following commutative diagram on \v{C}ech cohomologies
   \begin{equation*}
     \xymatrix{
       \check{H}^\bullet(X\cap\bigcup {\bf V}(\inf \phi^{-1}(t));{\bf V}(\inf \phi^{-1}(t))) & \check{H}^\bullet(X\cap\bigcup{\bf V}(\inf \phi^{-1}(t'));{\bf V}(\inf \phi^{-1}(t'))) \ar[l] \\
       \check{H}^\bullet(Y\cap\bigcup{\bf U}(t);{\bf U}(t)) \ar[u] & \check{H}^\bullet(Y\cap\bigcup{\bf U}(t');{\bf U}(t')) \ar[l] \ar[u] \\ 
       }
   \end{equation*}
   in which the horizontal homomorphisms arise from Lemma \ref{lem:cech_refinement} and the vertical homomorphisms arise from Lemma \ref{lem:cech_morphism}.  Each row can be interpreted as a persistence module and the vertical maps are morphisms in $\cat{PMod}$, which completes the argument.
 \end{proof}

 An immediate consequence of \cite[Prop. 4.8]{Harker_2018} is that if two persistence modules are $\epsilon$-interleaved (in the sense of Definition \ref{def:interleaving}), then the bottleneck distance between them is bounded above by $\epsilon$.  Taken with the proof of Proposition \ref{lem:cech_functor}, this implies the next statement.
 
 \begin{corollary}
   \label{cor:sctop_robustness}
   An $\epsilon$-interleaving between objects in $\cat{CoarseFilt}$, induces an $\epsilon$-interleaving between their persistent \v{C}ech cohomologies, so the bottleneck distance between their persistence diagrams is bounded above by $\epsilon$.
 \end{corollary}

 \section{Consistency radius on subspaces}

 The consistency radius of an assignment is a global property.  A large consistency radius may mean that there is widespread inconsistency among the values of an assignment, or the inconsistency may be localized to a small part of the base space.  To discriminate between these two situations, it is useful to restrict the consistency radius to a subset of the base space.
 
 \begin{definition}
   \label{def:cr_restricted}
   If $a$ is an assignment to a sheaf $\shf{S}$ of pseudometric spaces on a topological space $(X,\col{T})$, and $U \in \col{T}$ is open, then the \emph{local consistency radius on $U$} is
   \begin{equation*}
     c_{\shf{S}}(a,U) := c_{i^*_U \shf{S}}(i^*_U a),
   \end{equation*}
   where $i: U\to X$ is the inclusion.
 \end{definition}

 Intuitively, $c_{\shf{S}}(a,U)$ is only sensitive to the part of the assignment that touches $U$.

 \begin{proposition}
   \label{prop:cr_restricted}
   If $a$ is an assignment to a sheaf $\shf{S}$ of pseudometric spaces on a topological space $(X,\col{T})$, and $U \in \col{T}$ is open, then
   \begin{equation*}
    c_{\shf{S}}(a,U) = \sup_{V_1 \subseteq V_2 \subseteq U} d_{V_1}\left((\shf{S}(V_1 \subseteq V_2))a(V_2),a(V_1)\right).
   \end{equation*}
 \end{proposition}
 
 \begin{proof}
   The left side is by definition
     \begin{equation*}
       c_{i^*_U\shf{S}}(i^*_U a) = \sup_{V_1 \subseteq V_2 \in \col{T}} d_{V_1}\left(\left(i^*_U \shf{S}(V_1 \subseteq V_2)\right)(i^*_Ua(V_2)),(i^*_U a(V_1))\right),
     \end{equation*}
     where $i: U\to X$ is the inclusion map.
     Once we notice that $i^{-1}(V) = V$ for each open $V \subseteq U$, the result follows from three facts about pullbacks:
     \begin{enumerate}
     \item $i^*_U \shf{S}(V) = \shf{S}(V)$ for any open $V \subseteq U$,
     \item $i^*_U\shf{S}(V_1 \subseteq V_2) = \shf{S}(V_1 \subseteq V_2)$ for any $V_1 \subseteq V_2 \subseteq U$, and
     \item for any $V \subseteq U$,
       \begin{eqnarray*}
         (i^*_U)_V &:& i^*_U\shf{S}(i^{-1}(V)) \to \shf{S}(V) \\
         &:& i^*_U\shf{S}(V) \to \shf{S}(V) \\
         &:& \shf{S}(V) \to \shf{S}(V)
       \end{eqnarray*}
       is the identity map.  
     \end{enumerate}
 \end{proof}

 The local consistency radius on an open set is indeed local, in the sense that replacing the open set with a larger one cannot decrease the local consistency radius.

 \begin{lemma}
   \label{lem:cr_monotonicity}
   If $a$ is an assignment to a sheaf $\shf{S}$ of pseudometric spaces and $U \subseteq V$ are open subsets of the base space, then
   \begin{equation*}
     c_{\shf{S}}(a,U) \le c_{\shf{S}}(a,V).
   \end{equation*}
 \end{lemma}
 \begin{proof}
   This follows from the fact that the supremum in the expression for $c_{\shf{S}}(a,V)$ is over a strictly larger set than the supremum in the expression for $c_{\shf{S}}(a,U)$.
 \end{proof}

 \begin{corollary}
   \label{cor:cr_union}
   If $a$ is an assignment to a sheaf $\shf{S}$ of pseudometric spaces and $U, V$ are open subsets of the base space, then
   \begin{equation*}
     \max \{c_{\shf{S}}(a,U),c_{\shf{S}}(a,V)\} \le c_{\shf{S}}(a,U \cup V).
   \end{equation*}
 \end{corollary}

 Caution: Corollary \ref{cor:cr_union} does not ensure equality since $a(U \cup V)$ has no particular relationship to $a(U)$ or to $a(V)$.

 \begin{corollary}
   \label{cor:comb_monotonicity}
   For an assignment $a$ to a sheaf $\shf{S}$ of pseudometric spaces, the mean (or max, min, or any monotonic function of) consistency radius of a collection of open sets is monotonic with refinement.  That is, if $\col{U},\col{V}$ are collections of open sets in the base space of $\shf{S}$ and $\col{U}$ is a refinement of $\col{V}$, then the mean consistency radius for $\col{U}$ will be less than the mean consistency radius for $\col{V}$.
  \end{corollary}

  Lemma \ref{lem:cr_monotonicity} and its Corollaries indicate that the notion of consistency radius can be extended to an assignment that is only partially specified.

  \begin{definition}
    If $\col{U} \subseteq \col{T}$ is a collection of open sets for a topological space $(X,\col{T})$ and $\shf{S}$ is a sheaf of pseudometric spaces on $(X,\col{T})$, then an \emph{assignment supported on $\col{U}$} is an element of $\prod_{U \in \col{U}} \shf{S}(U)$.  The consistency radius of an assignment $a$ supported on $\col{U}$ is written $c_{\shf{S}}(a,\col{U})$, and is the infimum of all consistency radii of assignments $b$ that restrict to $a$, namely
    \begin{equation*}
      c_{\shf{S}}(a,\col{U}) := \inf \left\{c_{\shf{S}}(b) : b \in \prod_{V \in \col{T}}\shf{S}(V) \text{ such that } b(U) = a(U) \text{ whenever } U\in \col{U}\right\}.
    \end{equation*}
    We say that each such assignment $b$ \emph{extends} $a$.
  \end{definition}

  While this definition ensures that
  \begin{equation*}
    c_{\shf{S}}(a) = c_{\shf{S}}(a,X) = c_{\shf{S}}(a,\col{T}), 
  \end{equation*}
  the reader is warned that if an assignment $a$ is supported on a collection $\col{U} = \{U\}$ consisting of a single open set, then $c_{\shf{S}}(a)$ and $c_{\shf{S}}(a,U)$ are both meaningless if the base space topology $\col{T}$ is not trivial.  Nevertheless, the following Proposition is true.

  \begin{proposition}
    \label{prop:partial_monotonicity}
    If $a$ is an assignment to a sheaf $\shf{S}$ on a topological space $(X,\col{T})$ and $U \in \col{T}$, then
    \begin{equation*}
      c_{\shf{S}}(a,\col{T}\cap U) \ge c_{\shf{S}}(a,U).
    \end{equation*}
  \end{proposition}
  \begin{proof}
    Suppose $b$ is another assignment such that $b(V) = a(V)$ if $V \in \col{T} \cap U$.  Thus
    \begin{eqnarray*}
      c_{\shf{S}}(b) &=& \sup_{V_1 \subseteq V_2 \in \col{T}} d_{V_1}( \shf{S}(V_1 \subseteq V_2)b(V_2),b(V_1)) \\
      &\ge& \sup_{V_1 \subseteq V_2 \in \col{T} \cap U} d_{V_1}( \shf{S}(V_1 \subseteq V_2)b(V_2),b(V_1)) \\
      &\ge& \sup_{V_1 \subseteq V_2 \in \col{T} \cap U} d_{V_1}( \shf{S}(V_1 \subseteq V_2)a(V_2),a(V_1)) \\
      &\ge& \sup_{V_1 \subseteq V_2 \subseteq U} d_{V_1}( \shf{S}(V_1 \subseteq V_2)a(V_2),a(V_1)) \\
      &\ge& c_{\shf{S}}(a,U).
    \end{eqnarray*}
    Thus $c_{\shf{S}}(a,U)$ is a lower bound for the consistency radius of any extension $b$ of $a$.  Since $c_{\shf{S}}(a,\col{T}\cap U)$ is the greatest lower bound of these, the result follows.
  \end{proof}

  A nonzero consistency radius for an assignment that is supported on $\col{U}$ rather than the entire space is thus still the obstruction to extending that assignment to a global section.  It is for this reason that we need only consider assignments supported on the entire topology.  If we are given an assignment not supported on a given open set, a value can be supplied that minimizes the overall consistency radius.  Consistency is not assured over the entire base space by this process, so only a subspace will typically be covered by sets whose consistency radius is small.  The extent to which consistency is obtained is formalized by the following definition.

 \begin{definition}
   \label{def:consistent_cover}
   Let $a$ be an assignment to a sheaf $\shf{S}$ of pseudometric spaces on $(X,\col{T})$.  A collection of open sets $\col{U}$ is an \emph{$\epsilon$-consistent collection for $a$} if for every $U \in \col{U}$, the consistency radius on $U$ is less than $\epsilon$ 
   \begin{equation*}
     c_{\shf{S}}(a,U) < \epsilon.
   \end{equation*}
 \end{definition}

 Generally, $\epsilon$-consistent collections only cover part of the base space.

 \begin{example}
   \label{eg:partial_assignment}
   Consider the topological space $(X,\col{T})$ for the finite set $X:=\{A,B,C\}$ in which $\col{T}:=\{\emptyset,\{A\},\{A,B\},\{A,C\},\{A,B,C\}\}$.  We can define a sheaf $\shf{S}$ on $(X,\col{T})$ according to 
   \begin{equation*}
     \xymatrix{
       & \{A\}\ar[dl]\ar[dr]\ar[dd] & & & \mathbb{R} & \\
       \{A,B\}\ar[dr]&&\{A,C\}\ar[dl] & \mathbb{R} \ar[ur]^{1/2} & & \mathbb{R}\ar[ul]_{1}\\
       & \{A,B,C\} & & & \mathbb{R} \ar[ul]^{2r} \ar[uu]^{r} \ar[ur]_{r} & \\
       }
   \end{equation*}
   where the diagram on the left shows the open sets in $\col{T}$ and the diagram on the right shows the restrictions of $\shf{S}$.  (Each restriction of $\shf{S}$ is a homomorphism $\mathbb{R}\to\mathbb{R}$, which is given by multiplication by the listed factor.)  Observe that any sheafification of the sheaf on the partial order for the base $\{\{A\},\{A,B\},\{A,C\}\}$ will be of this form, with the $r\not=0$ being left as a free parameter.

   Consider the assignment $a$ supported on the two sets $\col{U}:=\{\{A,B\},\{A,C\}\}$ on the middle row above, given by
   \begin{equation*}
     a(\{A,B\}) := 0, \; a(\{A,C\}) := 1.
   \end{equation*}
   If we wish to compute $c_{\shf{S}}(a,\col{U})$, we need to extend to assignments over all open sets.  Namely, if
   \begin{equation*}
     a(\{A\}) := y, \; a(\{A,B,C\}) := x,
   \end{equation*}
   we must solve the optimization problem
   \begin{equation*}
     c_{\shf{S}}(a,\col{U}) = \min_{x,y} \max \{ |y|, |1-y|, |rx - y|, |2rx|, |1-rx| \}.
   \end{equation*}
   By inspection (or a short application of the simplex algorithm), we find that $y=1/2$ and $rx = 1/3$.  Thus the critical thresholds are independent of $r$ (and hence really only depend on the sheaf on the partial order rather than the topological space) and are given by
   \begin{equation*}
     c_{\shf{S}}(a,\col{U}) = \min_{x,y \in \mathbb{R}} \max \{ 1/2, 1/2, 1/6, 2/3, 2/3 \} = 2/3.
   \end{equation*}
   The consistency radius on each open set is shown in the left frame of Figure \ref{fig:cf_example}.
   
   Although there is an $2/3$-consistent collection that contains $X$, it happens that $X$ is covered by a $1/2$-consistent collection, namely $\col{U}$ itself.  If $\epsilon < 1/2$, the only consistent open set is $\{A\}$, which clearly does not cover $X$.  
 \end{example}

 The next Lemma provides a useful bound on the consistency radius of the image of an assignment through a morphism of $\cat{ShvPA}$.

  \begin{lemma}
   \label{lem:cr_morphism}
   Suppose $m:(\shf{S},a) \to (\shf{R},b)$ is a morphism in $\cat{ShvPA}$ along $f:X\to Y$, for which all component maps are Lipschitz continuous with constant $K$.  Then for every open $U \subseteq Y$,
   \begin{equation*}
     c_{\shf{R}}(b,U) \le K c_{\shf{S}}(a,f^{-1}(U)).
   \end{equation*}
 \end{lemma}
 \begin{proof}
   This is merely calculation:
   \begin{eqnarray*}
     c_{\shf{R}}(b,U) & = & \sup_{V_1 \subseteq V_2 \subseteq U} d_{V_1}\left((\shf{R}(V_1 \subseteq V_2))b(V_2),b(V_1)\right) \\
     & = & \sup_{V_1 \subseteq V_2 \subseteq U} d_{V_1}\left((\shf{R}(V_1 \subseteq V_2) \circ m_{V_2})a(f^{-1}(V_2)),m_{V_1}(a(f^{-1}(V_1)))\right) \\
     & = & \sup_{V_1 \subseteq V_2 \subseteq U} d_{V_1}\left((m_{V_1} \circ \shf{S}(f^{-1}(V_1) \subseteq f^{-1}(V_2)))a(f^{-1}(V_2)),m_{V_1}(a(f^{-1}(V_1)))\right) \\
     & \le & K \sup_{f^{-1}(V_1) \subseteq f^{-1}(V_2) \subseteq f^{-1}(U)} d_{f^{-1}(V_1)}\left(\shf{S}(f^{-1}(V_1) \subseteq f^{-1}(V_2))a(f^{-1}(V_2)),a(f^{-1}(V_1))\right) \\
     &\le& K c_{\shf{S}}(a,f^{-1}(U)).
   \end{eqnarray*}
 \end{proof}

 Lemma \ref{lem:cr_morphism} is a generalization of Proposition \ref{prop:shva_sections}, that morphisms in $\cat{ShvPA}$ preserve global sections, because a global section of $\shf{S}$ has consistency radius zero.

\section{Consistency filtrations}

Given an assignment to a sheaf of pseudometric spaces, there can be many possible $\epsilon$-consistent collections.  Of these, there is a distinguished $\epsilon$-consistent collection that consists of open sets that are as large as possible.

 \begin{lemma}
   \label{lem:consistent_cover}
   For every $\epsilon>0$ and every assignment $a$ to a sheaf $\shf{S}$ of pseudometric spaces on a finite space, there is a unique $\epsilon$-consistent collection $\col{M}_{\shf{S},a}(\epsilon)$ such that every other $\epsilon$-consistent collection is a refinement of $\col{M}_{\shf{S},a}(\epsilon)$.
 \end{lemma}
 This Lemma is not the same as the Theorem in \cite{Praggastis_2016}, where an assignment is made only to vertices on an abstract simplicial complex.
 \begin{proof}
   Lemma \ref{lem:cr_monotonicity} implies that any refinement of an $\epsilon$-consistent collection is itself $\epsilon$-consistent.  Because of the finiteness of the base space, if for some open $U$
   \begin{equation*}
     c_{\shf{S}}(a,U) < \epsilon,
   \end{equation*}
   then there is a unique maximal open set $U'$ containing $U$ for which
   \begin{equation*}
     c_{\shf{S}}(a,U') < \epsilon.
   \end{equation*}
   This argument applied to each open set completes the argument.
 \end{proof}

 \begin{definition}
   \label{def:consistency_filtration}
   The \emph{consistency filtration} ${\bf CF} : \cat{ShvFPA} \to \cat{CoarseFilt}$ associates the coarsening filtration
   \begin{equation*}
     \left({\bf CF}(\shf{S},a)\right)(\epsilon) := \col{M}_{\shf{S},a}(\epsilon)
   \end{equation*}
   to each sheaf assignment $(\shf{S},a)$ on a finite space, in which $\epsilon>0$ selects the threshold for the $\epsilon$-consistent collection $\col{M}_{\shf{S},a}(\epsilon)$.
 \end{definition}

 In what follows, we will show that the consistency filtration is \emph{both} a functor (Theorem \ref{thm:cf_functor}) \emph{and} a continuous function (Theorem \ref{thm:cf_robustness}), depending on the interpretation of $\cat{ShvFPA}$ and $\cat{CoarseFilt}$.  Specifically, although $\cat{ShvFPA}$ was defined as a category (Definition \ref{def:shva}), it is also a topological space according to Corollary \ref{cor:shvpa_is_a_space} (by taking the disjoint union over all topological spaces).  Similarly, although $\cat{CoarseFilt}$ was defined as a category (Definition \ref{def:sctop}), the interleaving distance is a pseudometric for its set of objects (Corollary \ref{cor:sctop_robustness}), making it into a pseudometric space.

   \begin{figure}
  \begin{center}
    \includegraphics[width=5in]{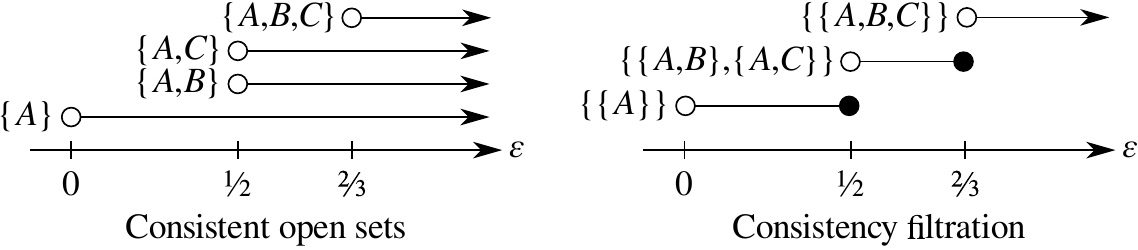}
    \caption{The consistency filtration for the assignment in Example \ref{eg:partial_assignment}; see Example \ref{eg:cf_example} for interpretation.}
    \label{fig:cf_example}
  \end{center}
 \end{figure}

   \begin{example}
     \label{eg:cf_example}
     The consistency filtration for the assignment in Example \ref{eg:partial_assignment} (after extending to an assignment supported on $\col{T}$) is given in Figure \ref{fig:cf_example}.  The left frame of Figure \ref{fig:cf_example} shows the ranges of $\epsilon$ for which each open set is $\epsilon$-consistent.  The frame on the right of Figure \ref{fig:cf_example} shows the consistency filtration, namely the coarsest $\epsilon$-consistent collections.  Observe that this is an element of $\cat{CoarseFilt}$ because $\{\{A\}\}$ refines $\{\{A,B\},\{A,C\}\}$ which refines $\{\{A,B,C\}\}$.

     This particular example has rather uninteresting persistent \v{C}ech cohomology, since $\check{H}^0 \cong \mathbb{R}$ and $\check{H}^1 = 0$ for all thresholds.
 \end{example}

 \begin{theorem}
   \label{thm:cf_functor}
   Consider $\cat{ShvFPA}_L$, the subcategory of $\cat{ShvFPA}$ consisting of sheaf morphisms along homeomorphisms whose component maps are Lipschitz.
   Consistency filtration is a covariant functor ${\bf CF}: \cat{ShvFPA}_L \to \cat{CoarseFilt}$.
 \end{theorem}

 \begin{proof}
First, we show that ${\bf CF}(\shf{S},a)$ is a coarsening filtration.  To this end, we need only observe that $\col{M}_{\shf{S},a}(t)$ is a refinement of $\col{M}_{\shf{S},a}(t')$ if $t < t'$ by Lemma \ref{lem:cr_monotonicity}.
   
Suppose $m:(\shf{S},a) \to (\shf{R},b)$ is a morphism in $\cat{ShvFPA}$ along the continuous map $f:(X,\col{T}_X) \to (Y,\col{T}_Y)$.  Since we assumed all component maps are Lipschitz and the topologies are finite, we may assume that the maximum Lipschitz constant of any component map is $K$.  Define
\begin{equation*}
  \phi(t) := \begin{cases}
    t & \text{if } K < 1,\\
    Kt & \text{otherwise.}
    \end{cases}
\end{equation*}
which is evidently order-preserving.

We need to argue that $\col{M}_{\shf{S},a}(\inf \phi^{-1}(t))$ refines $f^{-1}(\col{M}_{\shf{R},b}(t))$.  Let $V \in \col{M}_{\shf{S},a}(\inf \phi^{-1}(t))$, which means that $c_{\shf{S}}(a,V) < \inf \phi^{-1}(t) = t/K$.  (If $K < 1$, the bound is evidently $t$.  For brevity we assume $K \ge 1$ in what follows.)  But, due to the maximality of $\col{M}_{\shf{S},a}(\inf \phi^{-1}(t))$, any open set containing $V$ will have consistency radius larger than $t/K$.  Since $\st f(V)$ is open and $(Y,\col{T}_Y)$ is a finite topological space,
\begin{equation*}
  f^{-1}(\st f(V)) =  f^{-1}\left(\cap \{U \in \col{T}_Y : f(V) \subseteq U \}\right)
\end{equation*}
is an open set containing $V$, which means that
\begin{equation*}
  c_{\shf{S}}(a,f^{-1}(\st f(V))) \ge t/K \ge c_{\shf{S}}(a,V).
\end{equation*}
Because any $\col{T}_Y$-open set containing $f(V)$ will contain $\st f(V)$, this inequality means that $f(V)$ being contained in some $U\in \col{M}_{\shf{R},b}(t)$ implies that $V \subseteq f^{-1}(U)$.

Using the assumption that $f$ is a homeomorphism and that the Lipschitz constants of the component maps are less than or equal to 1, Lemma \ref{lem:cr_morphism} implies
\begin{equation*}
  c_{\shf{R}}(b,f(V)) \le K c_{\shf{S}}(a,V) \le t.
\end{equation*}
By the maximality of $\col{M}_{\shf{R},b}(t)$, this means that there is an open $U \in \col{M}_{\shf{R},b}(t)$ that contains $f(V)$.

Finally, we show that ${\bf CF}$ preserves composition of morphisms.  Start with two morphisms in $\cat{ShvFPA}_L$, $m:(\shf{S},a)\to(\shf{R},b)$ and $n:(\shf{R},b) \to (\shf{Q},c)$ along $f:(X,\col{T}_X) \to (Y,\col{T}_Y)$ and $g:(Y,\col{T}_Y) \to (Z,\col{T}_Z)$, respectively.  Using the above construction of ${\bf CF}(m)$, an $\cat{CoarseFilt}$ morphism for $m$, suppose that $\phi_m$ is the order preserving function constructed for $m$ so that $\col{M}_{\shf{S},a}(\inf \phi_m^{-1}(t))$ refines $f^{-1}(\col{M}_{\shf{R},b}(t))$.  Similarly, suppose that $\phi_n$ is the order preserving function constructed for ${\bf CF}(n)$ so that $\col{M}_{\shf{R},b}(\inf \phi_n^{-1}(t))$ refines $g^{-1}(\col{M}_{\shf{Q},c}(t))$.  Notice that the Lipschitz constants of the component maps of $m$ and $n$ impact our construction of $\phi_m$ and $\phi_n$, respectively.  We need to use these data to show that $\col{M}_{\shf{S},a}(\inf \phi_m^{-1}(\phi_n^{-1}(t)))$ refines $(g \circ f)^{-1}(\col{M}_{\shf{Q},c}(t))$.  Since $f$ is continuous,
\begin{equation*}
f^{-1}(\col{M}_{\shf{R},b}(\inf \phi_n^{-1}(t)))\text{ refines }f^{-1}(g^{-1}(\col{M}_{\shf{Q},c}(t))) = (g \circ f)^{-1}(\col{M}_{\shf{Q},c}(t)).
\end{equation*}
We know that
\begin{equation*}
\col{M}_{\shf{S},a}(\inf \phi_m^{-1}(\inf \phi_n^{-1}(t)))\text{ refines }f^{-1}(\col{M}_{\shf{R},b}(\inf \phi_n^{-1}(t))).
\end{equation*}
Since $\phi_m$ and $\phi_n$ are order preserving
\begin{equation*}
  \inf \phi_m^{-1}(\inf \phi_n^{-1}(t)) = \inf \phi_m^{-1}( \phi_n^{-1}(t)), 
\end{equation*}
which means that
\begin{equation*}
\col{M}_{\shf{S},a}(\inf \phi_m^{-1}(\phi_n^{-1}(t)))\text{ refines }(g \circ f)^{-1}(\col{M}_{\shf{Q},c}(t))
\end{equation*}
 as desired.
 \end{proof}

 \begin{remark}
   ${\bf CF}$ is not a faithful functor.  Global sections of a sheaf $\shf{S}$ may not be isomorphic objects in $\cat{ShvFPA}$, as Example \ref{eg:homeo_intransitive} shows, but they all have exactly the same consistency filtration.
 \end{remark}

 Recalling that we can interpret $\cat{ShvPA}(X,\col{T})$ and $\cat{CoarseFilt}$ not as categories but as topological spaces, the consistency filtration can be interpreted as a function.

 \begin{theorem}
   \label{thm:cf_robustness}
   If $(X,\col{T})$ is a finite topological space, then
   ${\bf CF}$ is a continuous function \\ $\cat{ShvPA}(X,\col{T}) \to \cat{CoarseFilt}$ under the interleaving distance.
 \end{theorem}

 Using Corollary \ref{cor:sctop_robustness}, this means that the transformation of isomorphism classes of $\cat{ShvPA}(X,\col{T})$ to persistence diagrams of \v{C}ech cohomology is continuous.

 \begin{proof}
   Since continuity is a local property, we work within each isomorphism class of $\cat{ShvPA}(X,\col{T})$ separately.
   
   Let $\epsilon > 0$ and $(\shf{S},a) \in \cat{ShvPA}(X,\col{T})$ be given.  We aim to show that there is an open neighborhood $Q \subseteq \cat{ShvPA}(X,\col{T})$ containing $(\shf{S},a)$ such that the consistency filtration of any $(\shf{R},b) \in Q$ is $\epsilon$-interleaved with the consistency filtration of $(\shf{S},a)$.
   
   Let $U \in \col{T}$ be an arbitrary open set.  Since $c_{\shf{S}}(a,U) = c_{i^*_U\shf{S}}(i^*_U a)$, this means that the local consistency radius on $U$ is a continuous function $\cat{ShvPA}(U,\col{T}\cap U) \to \mathbb{R}$, where $\col{T}\cap U$ is the subspace topology of $U$.  This means that there is an open $Q_U \subseteq \cat{ShvPA}(U,\col{T}\cap U)$ containing $(\shf{S},a)$ such that for every $(\shf{R},b) \in Q_U$, it follows that $|c_{\shf{R}}(b,U) - c_{\shf{S}}(a,U)|< \epsilon$.  This inequality still holds upon extending $Q_U$ to an open neighborhood $Q'_U$ in $\cat{ShvPA}(X,\col{T})$, since $c_{\shf{R}}(b,U)$ is independent of the sheaf and assignment outside $U$.  The set $Q'_U$ consists of those elements of 
   \begin{equation*}
     Q_U \times \left(\prod_{V' \subseteq V \in \col{T}, V \cap U=\emptyset} C(\shf{S}(V),\shf{S}(V')) \times \prod_{V \in \col{T}, V \cap U=\emptyset} \shf{S}(V)\right)
   \end{equation*}
   that actually correspond to sheaves in $\cat{ShvPA}(X,\col{T})$.  We are largely free to select spaces of local sections and restrictions outside of $U$ subject to commutativity of the diagram of restrictions and the gluing axiom.  Notice that we need only choose these for open sets $V$ that are disjoint from $U$, since the gluing axiom mandates the rest.
   Since $\col{T}$ is finite,
   \begin{equation*}
     Q := \bigcap_{U \in \col{T}} Q'_U
   \end{equation*}
   contains $(\shf{S},a)$ and is still open in $\cat{ShvPA}(X,\col{T})$.  Thus for any $(\shf{R},b) \in Q$, the consistency radius of every open set measured with $(\shf{R},b)$ differs from that measured by $(\shf{S},a)$ by no more than $\epsilon$.

   We now show that ${\bf CF}(\shf{S},a)$ and ${\bf CF}(\shf{R},b)$ are $\epsilon$-interleaved for any $(\shf{R},b) \in Q$.  Specifically, we construct morphisms $f: {\bf CF}(\shf{S},a) \to {\bf CF}(\shf{R},b)$ along the monotonic function $\phi$ and $g: {\bf CF}(\shf{R},b) \to {\bf CF}(\shf{S},a)$ along $\psi$ in $\cat{CoarseFilt}$ that are each others' inverses and $|\phi(t) - t| \le \epsilon$ and $|\psi(t) - t| \le \epsilon$.  This is easily done; let
   \begin{equation*}
     \phi(t) := \psi(t) = t + \epsilon,
   \end{equation*}
   and
   \begin{equation*}
     f_t := g_t := \id_X
   \end{equation*}
   for all $t \in \mathbb{R}$.  Without loss of generality, it remains to show that these are actually morphisms in $\cat{CoarseFilt}$, namely that
   \begin{equation*}
     \left({\bf CF}(\shf{S},a)\right)(\inf \phi^{-1}(t)) = \col{M}_{\shf{S},a}(t-\epsilon)
   \end{equation*}
   refines
   \begin{equation*}
     \left({\bf CF}(\shf{R},b)\right)(t) = f_t^{-1}(\col{M}_{\shf{R},b}(t)) = \col{M}_{\shf{R},b}(t)
   \end{equation*}
   for all $t\in \mathbb{R}$. Suppose that $U \in \col{M}_{\shf{S},a}(t-\epsilon)$, so that
   \begin{equation*}
     c_{\shf{S}}(a,U) < t-\epsilon.
   \end{equation*}
   This means that because $(\shf{R},b) \in Q$,
   \begin{eqnarray*}
     c_{\shf{R}}(b,U) &=& c_{\shf{R}}(b,U) + c_{\shf{S}}(a,U) - c_{\shf{S}}(a,U)\\
     &\le& c_{\shf{S}}(a,U) + | c_{\shf{R}}(b,U) - c_{\shf{S}}(a,U)| \\
     &<& c_{\shf{S}}(a,U) + \epsilon \\
     &<& t-\epsilon + \epsilon = t
   \end{eqnarray*}
   which implies that $U$ is an element (or a subset of some element) of $\col{M}_{\shf{R},b}(t)$.
    
 \end{proof}

 \begin{example}
   \label{eg:point_clouds}
   Consider a subset of $N$ points $X:=\{x_1, \dotsc, x_N\} \subseteq \mathbb{R}^M$ in $M$-dimensional Euclidean space.  We can realize $X$ as a partial assignment $x$ to the constant sheaf $\shf{K}$ of $\mathbb{R}^M$ on a topological space built from a single simplex.  To this end, let $Y := 2^X$ be the power set of $X$, which is also an abstract simplicial complex.  Given the partial order by inclusion $\subseteq$ on $Y$, form the Alexandrov topology $\col{T}$ on $Y$.  Specifically, since each set $U$ that is a star over $\{x_1, \dotsc x_k\}$ is open, let
   \begin{equation*}
     \shf{K}(\st \{x_1, \dotsc x_k\}) := \mathbb{R}^M
   \end{equation*}
   and let every restriction map be the identity map.  Let the assignment $x$ be supported on vertices (only), with
   \begin{equation*}
     x(\st \{x_i\}) := x_i.
   \end{equation*}

   Selecting the extension $a$ of $x$ that minimizes consistency radius,
   \begin{equation*}
     a = \text{argmin} \left\{c_{\shf{K}}(y) : y \in \prod_{U\in \col{T}} \mathbb{R}^M\text{ with } y(\st \{x_i\}) = x_i \text{ for each } x_i \in X\right\}
   \end{equation*}
   yields an assignment whose value on any $\col{T}$-open set $V$ is the circumcenter of that set of points in $X$ contained in $V$.  This is because the extension $a$ minimizes 
   \begin{equation*}
     \left\|a(\st \{x_1, \dotsc, x_k\}) - x_i\right\|
   \end{equation*}
   for all $i=1, \dotsc, k$.  Given this observation, a set $U:=\st \{x_1, \dotsc x_k\}$ is $\epsilon$-consistent if the intersection
   \begin{equation*}
     \bigcap_{i=1}^k B_\epsilon(x_i) 
   \end{equation*}
   of radius $\epsilon$ balls centered at each point is nonempty.  A maximal $\epsilon$-consistent collection $\col{M}_{\shf{K},a}(\epsilon)$ has cohomology isomorphic to the simplicial cohomology of the the radius $\epsilon$ \v{C}ech complex $C_\epsilon(X)$ of the point cloud, namely
   \begin{equation*}
     \check{H}^k(X \cap \bigcup \col{M}_{\shf{K},a}(\epsilon); \col{M}_{\shf{K},a}(\epsilon)) \cong H^k(C_\epsilon(X)).
   \end{equation*}
   Thus the \v{C}ech cohomology of the consistency filtration is isomorphic to the persistent \v{C}ech cohomology of the point cloud.
 \end{example}

 \section{Sheaves on partial orders}

 Much of the subtlety of the examples in this article arises from the fact that the value of an assignment on two open sets has no relationship to the value on their union.  While this is not a problem in principle, it presents a practical issue since sheaves on partial orders are generally most common in applications.  Sheaves on finite partial orders are convenient to specify since they merely record the stalks and restrictions on stars over points\footnote{Which is to say that their \emph{stalks} have the traditional meaning of a \emph{stalk of a sheaf} --- a direct limit.} and have no requirement aside from commutativity of the resulting diagram.  In particular, the gluing axiom is satisfied implicitly since the spaces of sections over unions of stars are not explicitly specified.  Practically, this means that a typical assignment for a sheaf on a partial order will be supported on the stars only.  As Example \ref{eg:partial_assignment} shows, the consistency radius for this kind of assignment may be strictly larger than the supremum of restrictions \emph{between stars}.  To mitigate this difficulty, we define the \emph{star consistency radius} for an open set.

 \begin{definition}
   For an assignment $a$ supported on the stars of a sheaf $\shf{S}$ of pseudometric spaces on an Alexandrov space $X$, the \emph{star consistency radius} on an open subset $U \subseteq X$ is given by
   \begin{eqnarray*}
     c^*_{\shf{S}}(a,U) &:= \max \{& \sup_{y \in U,} \sup_{x \in \st y} d_{\st x}\left(\left(\shf{S}(\st x \subseteq \st y)\right)a(\st y),a(\st x)\right), \\
     && \begin{aligned}\sup_{y \in U,} \sup_{x \in U,} \sup_{z \in \st x \cap \st y} \frac{1}{2}d_{\st z}\left(\left(\shf{S}(\st z \subseteq \st y)\right)a(\st y),\right. \\ \left.\left(\shf{S}(\st z \subseteq \st x)\right)a(\st x)\right) \}\end{aligned}
   \end{eqnarray*}
 \end{definition}

 \begin{lemma}
   \label{lem:star_consistency_bound}
   If $a$ is an assignment supported on the stars of a sheaf $\shf{S}$ of pseudometric spaces on an Alexandrov space $X$, then
   \begin{equation*}
     c^*_{\shf{S}}(a,U) \le c_{\shf{S}}(a,U).
   \end{equation*}
 \end{lemma}
 \begin{proof}
   Evidently the first quantity taken in the maximum (for the definition of star consistency radius) is less than the consistency radius since the consistency radius takes the supremum over a larger set.  The result follows from the calculation made in Remark \ref{rem:consistency_diameter} since the second quantity is a lower bound on the consistency diameter. 
 \end{proof}

 Without the second quantity in the maximum for the definition of star consistency of Lemma \ref{lem:star_consistency_bound}, the Lemma is clearly still true, but the bound is less tight.

 \begin{corollary}
   If $\col{U}$ is an $\epsilon$-consistent collection, then it is an $\epsilon$-star consistent collection as well.
 \end{corollary}

 Restricting our attention to sheaves on Alexandrov spaces and assignments supported on the set of all stars (and making appropriate changes to the arguments) does not impact the validity of essentially all of the results (Theorem \ref{thm:consistency_radius_continuous}, Lemma \ref{lem:cr_monotonicity}, Corollary \ref{cor:cr_union}, Corollary \ref{cor:comb_monotonicity}, Proposition \ref{prop:partial_monotonicity}, Lemma \ref{lem:cr_morphism}, Theorem \ref{thm:cf_functor}, and Theorem \ref{thm:cf_robustness}) obtained earlier in this article, though the proofs are somewhat more tedious.  

 \section*{Acknowledgements}
 This article is based upon work supported by the Defense Advanced Research Projects Agency (DARPA) and SPAWAR Systems Center Pacific (SSC Pacific) under Contract No. N66001-15-C-4040 and the Office of Naval Research (ONR) under Contract Nos. N00014-15-1-2090 and N00014-18-1-2541. Any opinions, findings and conclusions or recommendations expressed in this article are those of the author and do not necessarily reflect the views of the Defense Advanced Research Projects Agency (DARPA), SPAWAR Systems Center Pacific (SSC Pacific), or the Office of Naval Research (ONR).
 
\bibliographystyle{plainnat}
\bibliography{assignments_bib}
\end{document}